\documentclass{article}
\usepackage{fullpage}
\usepackage{amsthm}
\usepackage{amsmath}
\usepackage{amssymb}
\usepackage{authblk}
\usepackage{t1enc}
\usepackage[utf8]{inputenc}
\usepackage{colonequals}
\usepackage{float}
\usepackage{tikz}
\usetikzlibrary{calc,arrows.meta, decorations.shapes, shapes.geometric}
\usetikzlibrary{shapes,positioning,decorations.markings}
\usepackage{booktabs}
\usepackage{array}
\usepackage{multicol}
\usepackage{multirow}
\usepackage{caption}
\usepackage{cite}
\usepackage[nobysame]{amsrefs}
\usepackage{enumerate}
\usepackage{enumitem}
\usepackage[linesnumbered,ruled,vlined]{algorithm2e}
\SetArgSty{upshape}

\SetCommentSty{mycommfont}

\usepackage{titling}
\thanksmarkseries{arabic}

\usepackage[unicode,colorlinks=true,citecolor=green!40!black,linkcolor=red!20!black,urlcolor=blue!40!black,filecolor=cyan!30!black]{hyperref}

\makeatletter
 \newcommand{\linkdest}[1]{\Hy@raisedlink{\hypertarget{#1}{}}}
\makeatother

\usepackage[titles]{tocloft}
\renewcommand{\eprint}[1]{#1}
\renewcommand{\PrintYear}[1]{#1.}
\DefineSimpleKey{bib}{volumeadditionalinformation}
\DefineSimpleKey{bib}{issue}
\DefineSimpleKey{bib}{articlenumber}
\DefineSimpleKey{bib}{isbnnumber}

\BibSpec{article}{
 +{}  {\PrintAuthors}          {author}
 +{,} { \textit}               {title}
 +{,} { }                      {journal}
 +{}  { \textbf}               {volume}
 +{}  { \parenthesize}         {volumeadditionalinformation}
 +{}  {\parenthesize}          {issue}
 +{:} { }                      {pages}
 +{,} { article no.~}          {articlenumber}
 +{,} { \PrintYear}            {date}
 +{}  { Available at \eprint}  {eprint}
 +{}  { \PrintDOI}             {doi}
}

\BibSpec{inproceedings}{
 +{}  {\PrintAuthors}          {author}
 +{,} { \textit}               {title}
 +{.} { In \textit}            {booktitle}
 +{,} { }                      {series}
 +{,} { volume }               {volume}
 +{,} { pp.~}                  {pages}
 +{,} { article no.~}          {articlenumber}
 +{.} { }                      {publisher}
 +{,} { }                      {address}
 +{,} { \PrintYear}            {date}
 +{}  { \PrintDOI}             {doi}
}

\BibSpec{techreport}{
 +{}  {\PrintAuthors}          {author}
 +{,} { \textit}               {title}
 +{.} { }                      {series}
 +{,} { Technical Report No. } {number}
 +{.} { }                      {institution}
 +{,} { \PrintYear}            {date}
 +{}  { \PrintDOI}             {doi}
}

\BibSpec{book}{
 +{}  {\PrintAuthors}          {author}
 +{,} { \textit}               {title}
 +{,} { }                      {publisher}
 +{,} { }                      {address}
 +{,} { \PrintYear}            {date}
 +{}  { ISBN }                 {isbnnumber}
}

\allowdisplaybreaks

\theoremstyle{definition}
\newtheorem{theorem}{Theorem}[section]
\newtheorem{definition}[theorem]{Definition}
\newtheorem{lemma}[theorem]{Lemma}

\newtheorem{corollary}[theorem]{Corollary}

\title{\vspace{-50pt} Complexity of finding color-avoiding connected components}
\author{Panna T\'{i}mea Fekete\thanks{HUN-REN Alfr\'{e}d R\'{e}nyi Institute of Mathematics} \textsuperscript{,}\thanks{Department of Computer Science and Information Theory, Faculty of Electrical Engineering and Informatics, Budapest University of Technology and Economics} \and Roland Molontay\thanks{Department of Stochastics, Institute of Mathematics, Budapest University of Technology and Economics} \and Kitti Varga\protect\footnotemark[2]}

\begin{document}

\maketitle

\begin{abstract}
 Two vertices of a (not necessarily properly) arc-colored digraph are called arc-$\ell$-color-avoiding strongly $k$-arc-con\-nected if, after the removal of the arcs of any at most $\ell$ colors, they remain strongly $k$-arc-connected. The arc-$\ell$-color-avoiding strongly $k$-arc-connected components are maximal sets of vertices such that any two of those vertices belonging to the set are arc-$\ell$-color-avoiding strongly $k$-arc-connected in the graph.
 
 We investigate the problems of finding the arc-$\ell$-color-avoiding strongly $k$-arc-connected components and the maximal arc-$\ell$-color-avoiding strongly $k$-arc-connected subgraphs of an arc-colored digraph. We also consider possible generalizations to $k$-vertex-connectivity and to vertex-colored digraphs. We show that despite the seemingly minor differences in definitions, the complexity of these problems varies significantly: some can be solved in polynomial time, while others are NP-hard. It also turns out that in some cases the problem is NP-hard when multiple colors can be assigned to each vertex, but the problem becomes polynomial-time solvable when each vertex is single-colored.
\end{abstract}



\tableofcontents

\section{Introduction}

The robustness of networks against random errors and targeted attacks has been extensively studied in the past decades~\cites{albert2000error, posfai2016network, callaway2000network}. Understanding the attack and error tolerance of complex networks -- i.e., the ability to maintain the overall connectivity of the network as some vertices or edges are removed -- is a fundamental problem in network science, with applications in infrastructure resilience, communication security, and epidemic containment~\cites{newman2018networks, yuan2016k, callaway2000network}.

Various theoretical frameworks have been developed to analyze network robustness. One of the most widely used approaches is percolation theory, which studies the behavior of connected components when vertices or edges fail with a given probability~\cites{li2021percolation, cohen2021percolation}. Traditionally, percolation models assume that failures occur randomly (classical percolation) or follow a preferential removal strategy (targeted attacks) based on structural properties such as node degree or centrality measures~\cites{albert2000error, yuan2016k}. Several extensions of percolation have been proposed, including $k$-core percolation~\cite{dorogovtsev2006k}, bootstrap percolation~\cite{baxter2010bootstrap}, $l$-hop percolation~\cite{shang2011hop}, and percolation on multiplex networks~\cites{hackett2016bond, mahabadi2021network, son2012percolation, osat2017optimal}.

A significant limitation of standard network percolation models is that they treat failures as independent events, whereas in real-world systems, certain parts of the network may share common vulnerabilities. A notable example is transportation networks, where different transportation modes (rail, road, air, etc.) may be affected by distinct external factors. For instance, extreme weather conditions can disrupt railway transportation while leaving underground transport unaffected. Another example is communication networks, where routers within the same country may be compromised together, requiring an alternative routing strategy that avoids any single country's infrastructure.

While in the example of transportation networks, attacks target the edges; in the case of routers, they affect the vertices. In the latter scenario, we can consider two natural variations depending on how we define being compromised: whether the affected routers are entirely removed from the network and the question is how the remaining routers can still communicate; or whether routers registered to the same country can eavesdrop collectively on messages and the challenge is to ensure secure communication between two fixed routers despite potential interception.

To address such scenarios, Krause et al.~\cites{krause2016hidden, krause2017color, kadovic2018bond} introduced a novel framework called color-avoiding percolation. In this model, each vertex or edge is assigned a color representing a shared vulnerability. Two vertices are color-avoiding connected if they can reach each other via at least one path that avoids any single color. Color-avoiding percolation has received substantial research attention recently, both from the physics and mathematics communities~\cites{shekhtman2018critical, giusfredi2019self, giusfredi2020color, fekete2024color, lichev2024color, lichev2025color}.

Earlier, edge-color-avoiding colorings were studied under the name courteous edge-colorings by DeVos et al.~\cite{devos2006cut}. In particular, graphs that admit a 1-courteous edge-coloring are precisely the edge-color-avoiding connected graphs. That work established interesting upper bounds on the number of colors required to courteously color a given graph.
Molontay and Varga~\cite{molontay2019complexity} investigated the computational complexity of finding the so-called color-avoiding connected components of a graph. This concept was defined by Krause et al.~\cite{krause2016hidden} as maximal sets of vertices such that any two vertices belonging to the set are color-avoiding connected in the graph. Pintér and Varga~\cite{pinter2024color} studied the problem of finding color-avoiding connected spanning subgraphs with the minimum number of edges in edge- or vertex-colored graphs. They proved these problems to be NP-hard, proposed a polynomial-time approximation algorithm, and further extended the concept of edge-color-avoiding connectivity to matroids.

A wide variety of related problems have been studied across different research communities under different terminology. In network design, related formulations include shared-risk link and resource groups~\cites{cdp07,coudert16}, while the theoretical computer science literature has considered label-cut and colored-cut problems~\cites{labelcut,morawietz22}. Xu and Faragó~\cite{xu2019landscape} survey these overlapping formulations and highlight how closely related problems have often been rediscovered independently in different communities.

A particularly relevant formulation is the hedge cut problem, introduced by Ghaffari et al.~\cite{ghaffari2017random}, in which the objective is to determine the minimum number of colors whose removal disconnects a (not necessarily properly) edge-colored graph. The global version of this problem is also known as hedge connectivity. Ghaffari et al. gave a polynomial-time approximation scheme and a quasi-polynomial-time exact algorithm based on random sampling and contractions. Jaffke et al.~\cite{jaffke2023tight} subsequently established that the quasi-polynomial running time is essentially tight under the Randomized Exponential Time Hypothesis, while Fomin et al.~\cite{fomin2025fixed} showed that the problem is fixed-parameter tractable when parameterized by the number of removable colors. Related lines of research include color-fault-tolerant spanners and connectivity labeling schemes, which construct auxiliary structures that preserve connectivity under color failures~\cites{petruschka-spanners,petruschka-labeling}.

Typically, $k$-edge- or $k$-vertex-connected components are defined as maximal $k$-edge- or $k$-vertex-con\-nect\-ed subgraphs. However, this definition does not necessarily coincide with the concept of maximal sets of vertices such that any two vertices belonging to the set are $k$-edge- or $k$-vertex-connected in the graph, for arbitrary $k \in \mathbb{Z}_+$. The computational and algorithmic aspects of finding maximal $k$-edge- or $k$-vertex-connected subgraphs have been studied extensively; see, for example~\cites{holberg1992decomposition, li2017effective}. The problem of finding maximal sets of vertices such that any two vertices belonging to the set are $k$-edge-connected in the graph has also been investigated~\cites{wang2015simple}.

In this paper, we address these problems within the color-avoiding framework. 
We explore the concept of color-avoidance both for arc- and vertex-colored digraphs, including two variants for the latter, motivated by the real-world scenarios described above.
Our approach also considers the case in which multiple color classes are attacked simultaneously, and we require higher strong arc- or vertex-connectivity. We build upon the results of Molontay and Varga~\cite{molontay2019complexity} which settled the case $k = \ell = 1$ for finding the color-avoiding connected components, but here we use revised terminology.

The rest of the paper is organized as follows. In Section~\ref{section:preliminaries}, we collect some important definitions and notation. In Section~\ref{section:Alk}, we study the problems of finding color-avoiding strongly connected components and maximal color-avoiding strongly connected subgraphs in arc-colored digraphs. Sections~\ref{section:Vlk} and~\ref{section:IVlk} examine these problems for the two variants of color-avoidance in vertex-colored digraphs.

\bigskip

\section{Preliminaries} \label{section:preliminaries}

The set of positive integers is denoted by $\mathbb{Z}_+$. For any $n \in \mathbb{Z}_+$, let $[n] \colonequals \{ 1, \ldots, n \}$, and let $K_n$ denote the complete graph on $n$ vertices.

The \emph{complement} of a simple graph $G = (V, E)$ is denoted by $\overline{G}$. The subgraph of $G$ induced by some $V' \subseteq V$ is denoted by $G[V']$. For any graph $G$ and for any positive integer $p$, let $pG$ denote the disjoint union of $p$ copies of $G$. For any graphs $G$ and $H$, we say that the graph $G$ is \emph{$H$-free} if it does not contain $H$ as an induced subgraph.

A directed graph and a directed path are simply called a \emph{digraph} and a \emph{dipath}, respectively. A directed edge is called an \emph{arc}. Let $G = (V, A)$ be a digraph. For a set $A' \subseteq A$ of arcs or for a set $V' \subseteq V$ of vertices, we denote by $G-A'$ or by $G-V'$ the graph that is obtained by deleting the arcs of $A'$ or the vertices of $V'$ from $G$, respectively. 

We say that two vertices $u$ and $v$ of a digraph $G$ are \emph{strongly $k$-arc-connected} if there exist at least $k$ pairwise arc-disjoint $u$-$v$ dipaths, and at least $k$ pairwise arc-disjoint $v$-$u$ dipaths in $G$. Similarly, we say that $u$ and $v$ are \emph{strongly $k$-vertex-connected} if there exist at least $k$ pairwise internally vertex-disjoint $u$-$v$ dipaths, and at least $k$ pairwise internally vertex-disjoint $v$-$u$ dipaths in $G$.

A digraph is called \emph{strongly $k$-arc-connected} (or \emph{strongly $k$-vertex-connected}) if it has either at least two vertices and there are at least $k$ pairwise arc-disjoint dipaths (or at least $k$ pairwise internally vertex-disjoint dipaths) from any vertex to any other vertex, or if it has one vertex with at least $k$ loops on it (where $k \in \mathbb{Z}_+$).\footnote{Let us point out that the usual way to define $k$-vertex connectivity is to require the graph to have at least $k+1$ vertices and to remain connected whenever fewer than $k$ vertices are removed. The reason for the condition on the number of vertices in this definition is to be able to define the \emph{connectivity number} (that is, the maximum value of $k$ for which the graph is $k$-vertex-connected) of complete graphs. Note that with both versions of the definition, the complete graph $K_n$ is $(n-1)$-vertex-connected but not $n$-vertex-connected. However, with the usual definition, a 2-vertex graph with $k \ge 2$ parallel edges is not $k$-vertex-connected; whereas in this article, we consider it to be $k$-vertex-connected.}

In this paper, we study a generalized concept of connectivity to (not necessarily properly) colored graphs\footnote{A \emph{proper edge-coloring} (or \emph{proper vertex-coloring}) is an assignment of colors to the edges (or vertices) in which no two incident edges (or no two adjacent vertices) receive the same color.}. Let $C$ be a finite color set and $f \colon A \to 2^C$ or $f \colon V \to 2^C$ be a function assigning (possibly multiple or no) colors to the arcs or the vertices of a digraph $G = (V, A)$. Given a set $C' \subseteq C$ of colors, by the \emph{removal of $C'$} from $G$, we mean the removal of all arcs or vertices whose set of colors contains at least one color of $C'$.

Depending on whether the arcs or the vertices are colored and what kind of connectivity properties we require after the removal of any possible combinations of a given number of colors, we distinguish six possible generalizations of strong $k$-arc- or $k$-vertex-connectivity of two vertices. The so-obtained definitions are summarized in Table~\ref{table:summary_def_two_vertices}.

Considering any of the six concepts in Table~\ref{table:summary_def_two_vertices}, a digraph on at least two vertices is called \emph{color-avoiding strongly connected} if any two distinct vertices are color-avoiding strongly connected. (The case of one-vertex digraphs is addressed separately in the corresponding sections.) We define a \emph{color-avoiding strongly connected component} of a colored digraph as a maximal set of vertices such that any two of these vertices are color-avoiding strongly connected. A \emph{maximal color-avoiding strongly connected subgraph} of a colored digraph is defined as a maximal subgraph that is color-avoiding strongly connected. Later, we present some examples to show that these two concepts differ.

In our proofs, we make use of the following results. From a theorem of Balas and Yu~\cite{balas1989graphs}*{Theorem~4}, it follows that a $\overline{pK_2}$-free graph has $O(n^{2p-2})$ maximal cliques. Combining this with a theorem of Tsukiyama et al.~\cite{tsukiyama1977new}*{Theorem~3}, these maximal cliques can be found in polynomial time. More precisely, we obtain the following.

\begin{theorem}\label{thm:pK2_free_listing_max_clique}
 All the maximal cliques of a $\overline{pK_2}$-free graph can be found in $O(n^{2p+1})$ time for any $p \in \mathbb{Z}_+$.
\end{theorem}

\bigskip

{\small
\begin{center}
 \setlength{\tabcolsep}{2.75pt}
 \renewcommand{\arraystretch}{1.1}
 \captionsetup{type=table}
 \begin{tabular}{!{\vrule width 1.25pt} c !{\vrule width 1.25pt} c|c|c !{\vrule width 1.25pt}} \specialrule{1.25pt}{0pt}{0pt}
   & \multicolumn{2}{c}{\multirow{2}{*}{If after the removal of}} & \\[-5pt]
   & \multicolumn{3}{c !{\vrule width 1.25pt}}{} \\
   & \multirow{4}{*}{\parbox{0.22\textwidth}{\centering the arcs of any \\ at most $\ell$ colors $\ldots$}} & \multicolumn{2}{c !{\vrule width 1.25pt}}{\multirow{2}{*}{the vertices of any at most $\ell$ colors}} \\
   & & \multicolumn{2}{c !{\vrule width 1.25pt}}{} \\
   &  & \multirow{2}{*}{\parbox{0.22\textwidth}{\centering either at least one of \\ $u$ or $v$ is removed, or $\ldots$}} & \multirow{2}{*}{\parbox{0.22\textwidth}{\centering excluding $u$ and $v$ $\ldots$}} \\
   & & & \\ \specialrule{1.25pt}{0pt}{0pt}
  \multirow{3}{*}{\parbox{0.25\textwidth}{\centering $\ldots$ $u$ and $v$ remain \\ strongly $k$-arc-connected, \\ then $u$ and $v$ are }} & \multirow{3}{*}{\parbox{0.2\textwidth}{\centering $[\mathrm{A} \ell \, | \, k \mathrm{A}]$-connected.}} & \multirow{3}{*}{\parbox{0.2\textwidth}{\centering $[\mathrm{V} \ell \, | \, k \mathrm{A}]$-connected.}}& \multirow{3}{*}{\parbox{0.2\textwidth}{\centering $[\mathrm{IV} \ell \, | \, k \mathrm{A}]$-connected.}} \\
   & & & \\
   & & & \\ \hline
  \multirow{3}{*}{\parbox{0.25\textwidth}{\centering $\ldots$ $u$ and $v$ remain \\ strongly $k$-vertex-connected, \\ then $u$ and $v$ are}} & \multirow{3}{*}{\parbox{0.2\textwidth}{\centering  $[\mathrm{A} \ell \, | \, k \mathrm{V}]$-connected.}} & \multirow{3}{*}{\parbox{0.2\textwidth}{\centering $[\mathrm{V} \ell \, | \, k \mathrm{V}]$-connected.}} & \multirow{3}{*}{\parbox{0.2\textwidth}{\centering $[\mathrm{IV} \ell \, | \, k \mathrm{V}]$-connected.}} \\
   & & & \\
   & & & \\ \specialrule{1.25pt}{0pt}{0pt}
 \end{tabular}
 \captionof{table}{Summary of the definitions of the color-avoiding connectivity of two vertices presented in Definitions~\ref{def:AlkA_and_AlkV}, \ref{def:VlkA_and_VlkV}, and~\ref{def:IVlkA_and_IVlkV}.}
 \label{table:summary_def_two_vertices}
\end{center}
}

\bigskip

\medskip

In this article, we determine the complexity of finding color-avoiding strongly connected components and maximal color-avoiding strongly connected subgraphs of a colored digraph. To do so, we introduce an auxiliary graph and study its structure. Our results are summarized in Table~\ref{table:summary_of_results}. Note that in some cases, the complexity of these problems depends on whether every vertex is colored with a single color or with some subset of colors.

\begin{center}
\begin{table}[H]
\centering
\footnotesize
 \setlength{\tabcolsep}{5pt}
 \renewcommand{\arraystretch}{1.3}
\begin{tabular}{!{\vrule width 1.5pt} c|c !{\vrule width 1.5pt} c|c !{\vrule width 1.5pt}} \specialrule{1.5pt}{0pt}{0pt}
 \multirow{2}{*}{\textbf{\parbox{100pt}{\centering Type of color- \\ avoiding connectivity}}} & \multirow{2}{*}{\textbf{Type of coloring}} & \multirow{2}{*}{\textbf{\parbox{100pt}{\centering A property of \\ the auxiliary graph}}} & \multirow{2}{*}{\textbf{\parbox{100pt}{\centering Complexity of \\ the problems}}} \\
 & & & \\ \specialrule{1.5pt}{0pt}{0pt}
 \multirow{4}{*}{$[\mathrm{A} \ell \, | \, k \mathrm{A}]$-connectivity} & \multirow{4}{*}{$f \colon A(G) \to 2^C$} & \multirow{4}{*}{\parbox{100pt}{\centering $P_3$-free, \\ see Lemma~\ref{lemma:property_of_AlkA}}}  & \multirow{4}{*}{\parbox{100pt}{\centering in P, \\ see Thm.~\ref{thm:complexity_of_AlkA} and Cor.~\ref{cor:complexity_of_AlkA_2}}} \\
 & & & \\
 & & & \\
 & & & \\ \hline
 \multirow{4}{*}{$[\mathrm{A} \ell \, | \, k \mathrm{V}]$-connectivity} & \multirow{4}{*}{$f \colon A(G) \to 2^C$} & \multirow{4}{*}{\parbox{100pt}{\centering $\overline{\big( \lceil k/2 \rceil + 1 \big) K_2}$-free, \\ see Lemma~\ref{lemma:property_of_AlkV}}} & \multirow{4}{*}{\parbox{100pt}{\centering in P, \\ see Thm.~\ref{thm:complexity_of_AlkV} and Cor.~\ref{cor:complexity_of_AlkV_2}}} \\
 & & & \\
 & & & \\
 & & & \\ \specialrule{1pt}{0pt}{0pt}
 \multirow{4}{*}{$[\mathrm{V} \ell \, | \, k \mathrm{A}]$-connectivity} & \multirow{2}{*}{$f \colon V(G) \to 2^C$}& \multirow{2}{*}{\parbox{100pt}{\centering it can be arbitrary, \\ see the proof of Thm.~\ref{thm:complexity_of_VlkA_and_VlkV}}} & \multirow{2}{*}{\parbox{100pt}{\centering NP-complete, \\ see Thm.~\ref{thm:complexity_of_VlkA_and_VlkV} and Thm.~\ref{thm:complexity_of_VlkA_and_VlkV_2}}} \\ 
 & & & \\ \cline{2-4}
 & \multirow{2}{*}{$f \colon V(G) \to C$} & \multirow{2}{*}{\parbox{100pt}{\centering $\overline{(\ell + 2) K_2}$-free, \\ see Lemma~\ref{lemma:property_of_VlkA_single}}} & \multirow{2}{*}{\parbox{100pt}{\centering in P, \\ see Thm.~\ref{thm:complexity_of_VlkA_single} and Cor.~\ref{cor:complexity_of_VlkA_single_2}}} \\
 & & & \\ \hline
 \multirow{4}{*}{$[\mathrm{V} \ell \, | \, k \mathrm{V}]$-connectivity} & \multirow{2}{*}{$f \colon V(G) \to 2^C$} & \multirow{2}{*}{\parbox{100pt}{\centering it can be arbitrary, \\ see the proof of Thm.~\ref{thm:complexity_of_VlkA_and_VlkV}}} & \multirow{2}{*}{\parbox{100pt}{\centering NP-complete, \\ see Thm.~\ref{thm:complexity_of_VlkA_and_VlkV} and Thm.~\ref{thm:complexity_of_VlkA_and_VlkV_2}}} \\
 & & & \\ \cline{2-4}
 & \multirow{2}{*}{$f \colon V(G) \to C$} & \multirow{2}{*}{\parbox{105pt}{\centering $\overline{\big\lceil (k\ell + k + \ell + 2)/2 \big\rceil K_2}$-free, \\ see Lemma~\ref{lemma:property_of_VlkV_single}}} & \multirow{2}{*}{\parbox{100pt}{\centering in P, \\ see Thm.~\ref{thm:complexity_of_VlkV_single} and Cor.~\ref{cor:complexity_of_VlkV_single_2}}} \\
 & & & \\ \specialrule{1pt}{0pt}{0pt} 
 \multirow{4}{*}{$[\mathrm{IV} \ell \, | \, k \mathrm{A}]$-connectivity} & \multirow{4}{*}{$f \colon V(G) \to 2^C$} & \multirow{4}{*}{\parbox{100pt}{\centering it can be arbitrary \\ even if every vertex is of the same color, \\ see Thm.~\ref{thm:complexity_of_IVlkA_and_IVlkV}}} & \multirow{4}{*}{\parbox{100pt}{\centering NP-complete \\ even if every vertex is of the same color, \\ see Thm.~\ref{thm:complexity_of_IVlkA_and_IVlkV} and Thm.~\ref{thm:complexity_of_IVlkA_and_IVlkV_2}}} \\
 & & & \\
 & & & \\
 & & & \\ \hline
 \multirow{4}{*}{$[\mathrm{IV} \ell \, | \, k \mathrm{V}]$-connectivity} & \multirow{4}{*}{$f \colon V(G) \to 2^C$} & \multirow{4}{*}{\parbox{100pt}{\centering it can be arbitrary \\ even if every vertex is of the same color, \\ see the proof of Thm.~\ref{thm:complexity_of_IVlkA_and_IVlkV}}} & \multirow{4}{*}{\parbox{100pt}{\centering NP-complete \\ even if every vertex is of the same color, \\ see Thm.~\ref{thm:complexity_of_IVlkA_and_IVlkV} and Thm.~\ref{thm:complexity_of_IVlkA_and_IVlkV_2}}} \\
 & & & \\
 & & & \\
 & & & \\ \specialrule{1.5pt}{0pt}{0pt}
\end{tabular}
\caption{Summary of our results on the complexity of finding the color-avoiding strongly connected components and maximal color-avoiding strongly connected subgraphs for different concepts of color-avoiding connectivity (with fixed values of $k, \ell \in \mathbb{Z}_+$).}
\label{table:summary_of_results}
\end{table}
\end{center}

\section{Arc-color-avoiding strong connectivity} \label{section:Alk}

In this section, we focus on arc-color-avoiding strong connectivity. First, we present the formal definitions; then in Section~\ref{section:Alk_old_components}, we present an algorithm for finding the arc-color-avoiding strongly connected components of an arc-colored digraph. Finally, in Section~\ref{section:Alk_new_components}, we present an algorithm for finding the maximal arc-color-avoiding strongly connected subgraphs.

\begin{definition}[${[\mathrm{A} \ell \, | \, k \mathrm{A}]}$- and ${[\mathrm{A} \ell \, | \, k \mathrm{V}]}$-connectivity] \label{def:AlkA_and_AlkV}
 Let $k, \ell \in \mathbb{Z}_+$, let $G$ be a digraph, let $C$ be a finite color set, and let $f \colon A(G) \to 2^C$ be a function assigning some colors to every arc. Two distinct vertices $u, v \in V(G)$ are called \emph{arc-$\ell$-color-avoiding strongly $k$-arc-connected}, or for short \emph{$[\mathrm{A} \ell \, | \, k \mathrm{A}]$-connected} (and respectively, \emph{arc-$\ell$-color-avoiding strongly $k$-vertex-connected}, or for short \emph{$[\mathrm{A} \ell \, | \, k \mathrm{V}]$-connected}) if after the removal of any at most $\ell$ colors, they remain strongly $k$-arc-connected (or strongly $k$-vertex-connected).
\end{definition}

Considering each edge of an undirected graph as two oppositely directed arcs, the concept of arc-color-avoiding strong connectivity can also be applied to undirected graphs. In Figure \ref{fig:example_AlkA}, we can see an example for Definition~\ref{def:AlkA_and_AlkV}.

\begin{figure}[H]
\centering
\begin{tikzpicture}[scale=1.2]
 \tikzstyle{vertex}=[draw,circle,fill,minimum size=8,inner sep=0]
 \tikzset{edge_red/.style={postaction={decorate, decoration={markings, mark=between positions 3pt and 1-8pt step 7pt with {\draw[red!50!black,thin,fill=red!50] (30:0.08) -- (150:0.08) -- (210:0.08) -- (330:0.08) -- (30:0.08);}}}}}

 \tikzset{edge_blue/.style={postaction={decorate, decoration={markings, mark=between positions 5pt and 1-7pt step 6pt with {\draw[blue!50!black,thin,fill=blue!50] (0:0.08) -- (120:0.08) -- (240:0.08) -- (0:0.08);}}}}}
 
 \tikzset{edge_green/.style={postaction={decorate, decoration={markings, mark=between positions 5pt and 1-8pt step 8.5pt with {\draw[green!50!black,thin,fill=green!50] (0:0.16) -- (90:0.08) -- (180:0.08) -- (270:0.08) -- (0:0.16);}}}}}
 
 \tikzset{edge_yellow/.style={postaction={decorate, decoration={markings, mark=between positions 5pt and 1-8pt step 7pt with {\draw[yellow!50!black,thin,fill=yellow!50] (0:0.09) -- (72:0.09) -- (144:0.09) -- (216:0.09) -- (288:0.09) -- (0:0.09);}}}}}

 \tikzset{edge_red_and_blue/.style={postaction={decorate, decoration={markings, mark=between positions 3pt and 1-8pt step 11pt with {\draw[red!50!black,thin,fill=red!50] (30:0.08) -- (150:0.08) -- (210:0.08) -- (330:0.08) -- (30:0.08);}}}, postaction={decorate, decoration={markings, mark=between positions 8.5pt and 1-8pt step 11pt with {\draw[blue!50!black,thin,fill=blue!50] (0:0.08) -- (120:0.08) -- (240:0.08) -- (0:0.08);}}}}}

 \tikzset{edge_red_and_green/.style={postaction={decorate, decoration={markings, mark=between positions 3pt and 1-8pt step 14.5pt with {\draw[red!50!black,thin,fill=red!50] (30:0.08) -- (150:0.08) -- (210:0.08) -- (330:0.08) -- (30:0.08);}}}, postaction={decorate, decoration={markings, mark=between positions 9.5pt and 1-8pt step 14.5pt with {\draw[green!50!black,thin,fill=green!50] (0:0.16) -- (90:0.08) -- (180:0.08) -- (270:0.08) -- (0:0.16);}}}}}

 \tikzset{edge_blue_and_yellow/.style={postaction={decorate, decoration={markings, mark=between positions 10pt and 1-8pt step 12pt with {\draw[blue!50!black,thin,fill=blue!50] (0:0.08) -- (120:0.08) -- (240:0.08) -- (0:0.08);}}}, postaction={decorate, decoration={markings, mark=between positions 4pt and 1-8pt step 12pt with {\draw[yellow!50!black,thin,fill=yellow!50] (0:0.09) -- (72:0.09) -- (144:0.09) -- (216:0.09) -- (288:0.09) -- (0:0.09);}}}}}

 \node[vertex] (v1) at (0,0) [label=left:$v_1$] {};
 \node[vertex] (v2) at (1,1.25) [label=above:$v_2$] {};
 \node[vertex] (v3) at (1,-1.5) [label=below:$v_3$] {};
 \node[vertex] (v4) at (2,0) [label=right:$v_4$] {};
 
 \path[edge_red_and_blue] (v1) -- (v2);
 \draw[->,>={LaTeX[black,length=10pt, width=10pt]},draw opacity=0] (v1) -- (v2);
 \path[edge_blue] (v1) to [bend left=15] (v3);
 \draw[->,>={LaTeX[black,length=10pt, width=10pt]},draw opacity=0] (v1) to [bend left=15] (v3);
 \path[edge_green] (v1) to [bend left=15] (v4);
 \draw[->,>={LaTeX[black,length=10pt, width=10pt]},draw opacity=0] (v1) to [bend left=15] (v4);
 \path[edge_red] (v2) -- (v4);
 \draw[->,>={LaTeX[black,length=10pt, width=10pt]},draw opacity=0] (v2) -- (v4);
 \path[edge_green] (v3) to [bend left=15] (v1);
 \draw[->,>={LaTeX[black,length=10pt, width=10pt]},draw opacity=0] (v3) to [bend left=15] (v1);
 \path[edge_blue_and_yellow] (v4) to [bend left=15] (v1);
 \draw[->,>={LaTeX[black,length=10pt, width=10pt]},draw opacity=0] (v4) to [bend left=15] (v1);
 \path[edge_red_and_green] (v4) -- (v3);
 \draw[->,>={LaTeX[black,length=10pt, width=10pt]},draw opacity=0] (v4) -- (v3);
\end{tikzpicture}
\caption{Examples for $[\mathrm{A} \ell \, | \, k \mathrm{A}]$-connected vertices. The colors red, blue, green, and yellow are denoted by (red) rectangles, (blue) triangles, (green) deltoids, and (yellow) pentagons, respectively. The arcs marked with alternating forms are assigned multiple colors. The vertices $v_1$ and $v_4$ are $[\mathrm{A} 1 \, | \, 1 \mathrm{A}]$-connected, but the vertices $v_1$ and $v_3$ are not: after the removal of the green arcs, there remains no $v_3$-$v_1$ dipath.}
\label{fig:example_AlkA}
\end{figure}
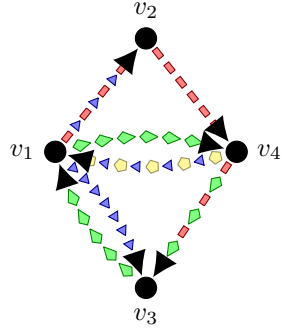

\noindent It follows directly from Menger's classical theorems~\cite{menger1927allgemeinen} that
 \begin{enumerate}[topsep=1.5pt, itemsep=-1.5pt, label=(\alph*)]
  \item \label{item:AlkA} \linkdest{item:AlkA_of_lemma}{} the vertices $u$ and $v$ of an arc-colored digraph $G$ are not $[\mathrm{A} \ell \mid k \mathrm{A}]$-connected if and only if there exist $C' \subseteq C$ and $A' \subseteq A(G)$ with $|C'| \le \ell$ and $|A'| \le k-1$ whose removal from $G$ does not leave any $u$-$v$ or any $v$-$u$ dipath;
  
  \item \label{item:AlkV} \linkdest{item:AlkV_of_lemma}{} the vertices $u$ and $v$ of an arc-colored digraph $G$ are not $[\mathrm{A} \ell \, | \, k \mathrm{V}]$-connected if and only if there exist $C' \subseteq C$ and $A' \subseteq A(G)$ and $V' \subseteq V(G) \setminus \{ u, v \}$ with $|C'| \le \ell$ and $|A'| + |V'| \le k-1$ whose removal from $G$ does not leave any $u$-$v$ or any $v$-$u$ dipath.
 \end{enumerate}

\smallskip

Also observe that the notion of $[\mathrm{A} \ell \, | \, 1 \mathrm{A}]$- and $[\mathrm{A} \ell \, | \, 1 \mathrm{V}]$-connectivity is the same. Finally, note that if every arc of a digraph is colored with zero colors, then two vertices are $[\mathrm{A} \ell \, | \, k \mathrm{A}]$-connected (and respectively, $[\mathrm{A} \ell \, | \, k \mathrm{V}]$-connected) if and only if they are strongly $k$-arc-connected (or strongly $k$-vertex-connected).

\subsection{Arc-color-avoiding strongly connected components} \label{section:Alk_old_components}

In this section, we define arc-color-avoiding strongly connected components. In Sections~\ref{section:AlkA_old_components} and~\ref{section:AlkV_old_components}, we show that the $[\mathrm{A} \ell \, | \, k \mathrm{A}]$- and the $[\mathrm{A} \ell \, | \, k \mathrm{V}]$-connected components of an arc-colored digraph can be found in polynomial time, respectively.

Let us point out that $[\mathrm{A} \ell \, | \, k \mathrm{A}]$-connectivity is an equivalence relation on the vertex set of the digraph for any $k, \ell \in \mathbb{Z_+}$: reflexivity and symmetry are trivial (if considering each vertex being $[\mathrm{A} \ell \, | \, k \mathrm{A}]$-connected to itself), and now we show that transitivity also holds. Assume that the vertices $u$ and $v$ are $[\mathrm{A} \ell \, | \, k \mathrm{A}]$-connected and so are the vertices $v$ and $w$. Then after the removal of any at most $\ell$ colors and any at most $k-1$ arcs, there remains a $u$-$v$, a $v$-$u$, a $v$-$w$, and a $w$-$v$ dipath; thus $u$ and $w$ are also $[\mathrm{A} \ell \, | \, k \mathrm{A}]$-connected. The equivalence classes of $[\mathrm{A} \ell \, | \, k \mathrm{A}]$-connectivity are called \emph{arc-$\ell$-color-avoiding strongly $k$-arc-connected components}, or for short \emph{$[\mathrm{A} \ell \, | \, k \mathrm{A}]$-connected components}.

However, the relation of $[\mathrm{A} \ell \, | \, k \mathrm{V}]$-connectivity is not transitive for any $k \ge 2$: when every arc is colored with zero colors, then $[\mathrm{A} \ell \, | \, k \mathrm{V}]$-connectivity is the same as strong $k$-vertex-connectivity, and it is well known that the latter is not transitive for any $k \ge 2$. Thus, $[\mathrm{A} \ell \, | \, k \mathrm{V}]$-connectivity is not an equivalence relation.

For any $k, \ell \in \mathbb{Z}_+$, the \emph{arc-$\ell$-color-avoiding strongly $k$-vertex-connected components}, or for short the \emph{$[\mathrm{A} \ell \, | \, k \mathrm{V}]$-connected components}, of an arc-colored digraph $G$ are maximal sets of vertices such that any two vertices belonging to the set are $[\mathrm{A} \ell \, | \, k \mathrm{V}]$-connected in $G$. If a vertex $v \in V(G)$ is not $[\mathrm{A} \ell \, | \, k \mathrm{A}]$- or $[\mathrm{A} \ell \, | \, k \mathrm{V}]$-connected to any other vertex of the digraph $G$ for some $k, \ell \in \mathbb{Z}_+$, then $v$ is considered to form an $[\mathrm{A} \ell \, | \, k \mathrm{A}]$- or an $[\mathrm{A} \ell \, | \, k \mathrm{V}]$-connected component of $G$ in itself, respectively.

We say that an arc-colored digraph on at least two vertices is \emph{arc-$\ell$-color-avoiding strongly $k$-arc-connected}, or for short $[\mathrm{A} \ell \, | \, k \mathrm{A}]$-connected (and respectively, \emph{arc-$\ell$-color-avoiding strongly $k$-vertex-connected}, or for short $[\mathrm{A} \ell \, | \, k \mathrm{V}]$-connected) if it has exactly one $[\mathrm{A} \ell \, | \, k \mathrm{A}]$-connected component (or $[\mathrm{A} \ell \, | \, k \mathrm{V}]$-connected component). An arc-colored digraph on exactly one vertex is $[\mathrm{A} \ell \, | \, k \mathrm{A}]$-connected (and respectively, $[\mathrm{A} \ell \, | \, k \mathrm{V}]$-connected) if, after the removal of any at most $\ell$ colors, it has at least $k$ loops.\footnote{That is, an isolated vertex in a digraph $G$ is considered to form an $[\mathrm{A} \ell \, | \, k \mathrm{A}]$-connected component (and a $[\mathrm{A} \ell \, | \, k \mathrm{V}]$-connected component) of $G$, even though the spanned subgraph $G[v]$ is not $[\mathrm{A} \ell \, | \, k \mathrm{A}]$-connected (and not $[\mathrm{A} \ell \, | \, k \mathrm{V}]$-connected).}

To find the arc-color-avoiding strongly connected components of a digraph, we use auxiliary graphs.

\begin{definition}[Auxiliary graphs $G_{[\mathrm{A} \ell \mid k \mathrm{A}]}$ and $G_{[\mathrm{A} \ell \mid k \mathrm{V}]}$]
 Let $k, \ell \in \mathbb{Z}_+$, let $G$ be a digraph, let $C$ be a finite color set, and let $f \colon A(G) \to 2^C$. The \emph{auxiliary graphs} of $G$ are the undirected graphs $G_{[\mathrm{A} \ell \mid k \mathrm{A}]}$ and $G_{[\mathrm{A} \ell \mid k \mathrm{V}]}$ on the same vertex set as $G$ and two vertices of theirs are adjacent if and only if the corresponding vertices in $G$ are $[\mathrm{A} \ell \, | \, k \mathrm{A}]$- or $[\mathrm{A} \ell \, | \, k \mathrm{V}]$-connected, respectively.
\end{definition}

First, note that $W \subseteq V(G)$ is the vertex set of an $[\mathrm{A} \ell \, | \, k \mathrm{A}]$- or an $[\mathrm{A} \ell \, | \, k \mathrm{V}]$-connected component of $G$ if and only if $W$ spans a maximal clique in $G_{[\mathrm{A} \ell \mid k \mathrm{A}]}$ or $G_{[\mathrm{A} \ell \mid k \mathrm{V}]}$, respectively.

Also, observe that for fixed $k, \ell \in \mathbb{Z}_+$, we can construct the auxiliary graphs $G_{[\mathrm{A} \ell \mid k \mathrm{A}]}$ and $G_{[\mathrm{A} \ell \mid k \mathrm{V}]}$ in polynomial time by checking for every pair of vertices whether there exist at least $k$ pairwise arc- or internally vertex-disjoint dipaths between them in both directions after the removal of $C'$ from $G$ for each $C' \subseteq C$ with $|C'| \le \ell$. This can be done by running a maximum $s$-$t$ flow algorithm for any $s,t \in V(G)$ and for any subset $C' \subseteq C$ with $s \ne t$ and $|C'| \le \ell$.

For some examples on $[\mathrm{A} \ell \, | \, k \mathrm{A}]$- or $[\mathrm{A} \ell \, | \, k \mathrm{V}]$-connected components, see Figure~\ref{fig:example_AlkA_2}.

\begin{figure}[H]
\centering
\begin{tikzpicture}[scale=1]
 \tikzstyle{vertex}=[draw,circle,fill,minimum size=8,inner sep=0]
 \tikzstyle{vertex_small}=[draw,circle,fill,minimum size=6,inner sep=0]
 \tikzset{edge_red/.style={postaction={decorate, decoration={markings, mark=between positions 3pt and 1-2pt step 7pt with {\draw[red!50!black,thin,fill=red!50] (30:0.08) -- (150:0.08) -- (210:0.08) -- (330:0.08) -- (30:0.08);}}}}}

 \tikzset{edge_blue/.style={postaction={decorate, decoration={markings, mark=between positions 2pt and 1-2pt step 11pt with {\draw[blue!50!black,thin,fill=blue!50] ($(90:0.08) - (0,0.02)$) -- ($(210:0.08) - (0,0.02)$) -- ($(330:0.08) - (0,0.02)$) -- ($(90:0.08) - (0,0.02)$);}}}, postaction={decorate, decoration={markings, mark=between positions 7.5pt and 1-2pt step 11pt with {\draw[blue!50!black,thin,fill=blue!50] ($(270:0.08) + (0,0.02)$) -- ($(30:0.08) + (0,0.02)$) -- ($(150:0.08) + (0,0.02)$) -- ($(270:0.08) + (0,0.02)$);}}}}}
 
 \tikzset{edge_green/.style={postaction={decorate, decoration={markings, mark=between positions 3pt and 1-3pt step 9pt with {\draw[green!50!black,thin,fill=green!50] (0:0.16) -- (90:0.08) -- (180:0.08) -- (270:0.08) -- (0:0.16);}}}}}

 \node at (-0.25,1.25) {$G$};
 
 \node[vertex] (v1) at (0,0) [label=left:$v_1$] {};
 \node[vertex] (v2) at (1,1) [label=above:$v_2$] {};
 \node[vertex] (v3) at (1,-1) [label=below:$v_3$] {};
 \node[vertex] (v4) at (2,0) [label=below:$v_4$] {};
 \node[vertex] (v5) at (3,1) [label=above:$v_5$] {};
 \node[vertex] (v6) at (3,-1) [label=below:$v_6$] {};
 \node[vertex] (v7) at (4,0) [label=right:$v_7$] {};

 \path[edge_red] (v1) -- (v2);
 \path[edge_blue] (v1) -- (v3);
 \path[edge_green] (v1) -- (v4);
 \path[edge_green] (v2) -- (v3);
 \path[edge_blue] (v2) -- (v4);
 \path[edge_red] (v3) -- (v4);
 \path[edge_red] (v4) -- (v5);
 \path[edge_blue] (v4) -- (v6);
 \path[edge_green] (v4) -- (v7);
 \path[edge_red] (v5) -- (v7);
 \path[edge_blue] (v6) -- (v7);
 
 \begin{scope}[shift={(7,0)}]
 \node at (-0.25,1.25) {$G_{[\mathrm{A} 1 \mid 2 \mathrm{V}]}$};
 
 \node[vertex_small] (v1) at (0,0) [label=left:$v_1$] {};
 \node[vertex_small] (v2) at (1,1) [label=above:$v_2$] {};
 \node[vertex_small] (v3) at (1,-1) [label=below:$v_3$] {};
 \node[vertex_small] (v4) at (2,0) [label=below:$v_4$] {};
 \node[vertex_small] (v5) at (3,1) [label=above:$v_5$] {};
 \node[vertex_small] (v6) at (3,-1) [label=below:$v_6$] {};
 \node[vertex_small] (v7) at (4,0) [label=right:$v_7$] {};

 \draw[very thick] (v1) -- (v2);
 \draw[very thick] (v1) -- (v3);
 \draw[very thick] (v1) -- (v4);
 \draw[very thick] (v2) -- (v3);
 \draw[very thick] (v2) -- (v4);
 \draw[very thick] (v3) -- (v4);
 \draw[very thick] (v4) -- (v7);
 \end{scope}
\end{tikzpicture}
\caption{Examples for $[\mathrm{A} \ell \, | \, k \mathrm{A}]$- and $[\mathrm{A} \ell \, | \, k \mathrm{V}]$-connected components. An edge-colored graph $G$ is shown on the left, and the auxiliary graph $G_{[\mathrm{A} 1 \mid 2 \mathrm{V}]}$ is shown on the right. The vertex sets of the $[\mathrm{A} 1 \, | \, 2 \mathrm{V}]$-connected components of $G$ are $\{ v_1, v_2, v_3, v_4 \}$, $\{ v_4, v_7 \}$, $\{ v_5 \}$, and $\{ v_6 \}$. The auxiliary graphs $G_{[\mathrm{A} 1 \mid 1 \mathrm{A}]} = G_{[\mathrm{A} 1 \mid 1 \mathrm{V}]}$ and $G_{[\mathrm{A} 1 \mid 2 \mathrm{A}]}$ of the undirected graph $G$ consist of a clique on the vertices $v_1, v_2, v_3, v_4, v_7$ and two isolated vertices $v_5$ and $v_6$. The auxiliary graph $G_{[\mathrm{A} 2 \mid 1 \mathrm{A}]}$ consists of seven isolated vertices only.}
\label{fig:example_AlkA_2}
\end{figure}
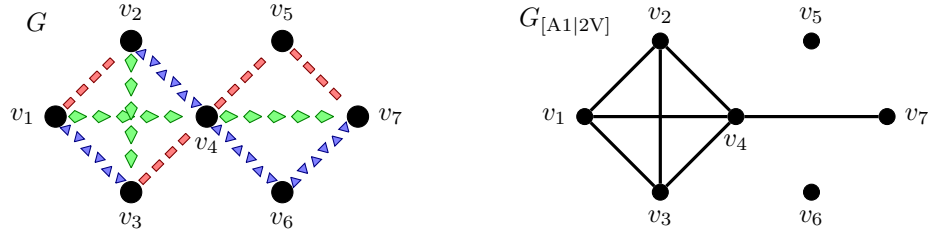

\subsubsection[\texorpdfstring{The $[\mathrm{A} \ell \, | \, k \mathrm{A}]$}{}-connected components]{\texorpdfstring{The $\boldsymbol{[\mathrm{A} \ell \, | \, k \mathrm{A}]}$}{}-connected components} \label{section:AlkA_old_components}

Since $[\mathrm{A} \ell \, | \, k \mathrm{A}]$-connectivity is an equivalence relation, we conclude the following.

\begin{lemma} \label{lemma:property_of_AlkA}
 Let $k, \ell \in \mathbb{Z}_+$, let $G$ be a digraph, let $C$ be a finite color set, and let $f \colon A(G) \to 2^C$. Then the auxiliary graph $G_{[\mathrm{A} \ell \mid k \mathrm{A}]}$ is $P_3$-free, i.e., it does not contain an induced path on three vertices.
\end{lemma}

Since the maximal cliques of a $P_3$-free graph are its components, and the components of any undirected graph can be found in polynomial time, for example, by using breadth-first search, we conclude that for any fixed $k, \ell \in \mathbb{Z}_+$, the $[\mathrm{A} \ell \, | \, k \mathrm{A}]$-connected components can be found in polynomial time.

\begin{theorem} \label{thm:complexity_of_AlkA}
 For any $k, \ell \in \mathbb{Z}_+$, the following problem can be solved in polynomial time.

 \smallskip
 
 \noindent {\bf \scshape \textsc{$\boldsymbol{[\mathrm{A} \ell \, | \, k \mathrm{A}]}$-ConnectedComponents}} \\
 \textit{Instance:} a digraph $G$, a finite color set $C$, a function $f \colon A(G) \to 2^C$, and $m \in \mathbb{Z}_+$. \\
 \textit{Question:} is it true that $G$ has an $[\mathrm{A} \ell \, | \, k \mathrm{A}]$-connected component of size at least $m$?

 \smallskip
 
 \noindent Moreover, all the $[\mathrm{A} \ell \, | \, k \mathrm{A}]$-connected components of $G$ can be found in polynomial time.
\end{theorem}

\subsubsection[\texorpdfstring{The $[\mathrm{A} \ell \, | \, k \mathrm{V}]$}{}-connected components]{\texorpdfstring{The $\boldsymbol{[\mathrm{A} \ell \, | \, k \mathrm{V}]}$}{}-connected components} \label{section:AlkV_old_components}

Even though $[\mathrm{A} \ell \, | \, k \mathrm{V}]$-connectivity is not an equivalence relation, we obtain the following.

\begin{lemma} \label{lemma:property_of_AlkV}
 Let $k, \ell \in \mathbb{Z}_+$, let $G$ be a digraph, let $C$ be a finite color set, and let $f \colon A(G) \to 2^C$. Then the auxiliary graph $G_{[\mathrm{A}\ell \mid k \mathrm{V}]}$ is $\overline{p K_2}$-free, where $p = \lceil k/2 \rceil + 1$.
\end{lemma}
\begin{proof}
 Suppose to the contrary that $G_{[\mathrm{A} \ell \mid k \mathrm{V}]}$ contains an induced $\overline{p K_2}$ subgraph. Let $v_1, \ldots, v_{2p} \! \in \! V \big( G_{[\mathrm{A} \ell \mid k \mathrm{V}]} \big)$ denote the vertices of this induced $\overline{p K_2}$ such that $v_{2i-1} v_{2i} \notin E \big( \, \overline{p K_2} \, \big)$ for any $i \in [p]$.

 By a previous observation regarding Definition~\ref{def:AlkA_and_AlkV}, see~\ref{item:AlkV}, since $v_{2p-1} v_{2p} \notin E \big( G_{[\mathrm{A} \ell \mid k \mathrm{V}]} \big)$, there exist $C' \subseteq C$ and $A' \subseteq A \big( G \big)$ and $V' \subseteq V \big( G \big) \setminus \{ v_{2p-1}, v_{2p} \}$ with $|C'| \le \ell$ and $|A'| + |V'| \le k-1$ such that after the removal of $C'$, $A'$, and $V'$ from $G$ there does not remain any $v_{2p-1}$-$v_{2p}$ or any $v_{2p}$-$v_{2p-1}$ dipath. By the pigeonhole principle, at least one of the vertices $v_1, \ldots, v_{2p-2}$ is not in $V'$; without loss of generality, we can assume that $v_1$ is such a vertex. However, since $v_{2p-1} v_1, v_1 v_{2p} \in E \big( G_{[\mathrm{A} \ell \mid k \mathrm{V}]} \big)$, after the removal of $C'$, $A'$, and $V'$ from $G$ there remains a $v_{2p-1}$-$v_1$, a $v_1$-$v_{2p-1}$, a $v_1$-$v_{2p}$, and a $v_{2p}$-$v_1$ dipath. Hence, there also remains a directed walk between $v_{2p-1}$ and $v_{2p}$ in both directions via $v_1$, which is a contradiction.
\end{proof}

By Theorem~\ref{thm:pK2_free_listing_max_clique}, we conclude that for any fixed $k, \ell \in \mathbb{Z}_+$, the $[\mathrm{A} \ell \, | \, k \mathrm{V}]$-connected components can be found in polynomial time.

\begin{theorem} \label{thm:complexity_of_AlkV}
 For any $k, \ell \in \mathbb{Z}_+$, the following problem can be solved in polynomial time.

 \smallskip
 
 \noindent {\bf \textsc{$\boldsymbol{[\mathrm{A} \ell \, | \, k \mathrm{V}]}$-ConnectedComponents}}

 \smallskip
 
 \noindent \textit{Instance:} a digraph $G$, a finite color set $C$, a function $f \colon A(G) \to 2^C$, and $m \in \mathbb{Z}_+$.
 
 \noindent \textit{Question:} is it true that $G$ has an $[\mathrm{A} \ell \, | \, k \mathrm{V}]$-connected component of size at least~$m$?

 \smallskip
 
 \noindent Moreover, all the $[\mathrm{A} \ell \, | \, k \mathrm{V}]$-connected components of $G$ can be found in polynomial time.
\end{theorem}

\subsection{Maximal arc-color-avoiding strongly connected subgraphs} \label{section:Alk_new_components}

In this section, we define maximal arc-color-avoiding strongly connected subgraphs. Then, in Sections~\ref{section:AlkA_new_components} and~\ref{section:AlkV_new_components}, we show that the maximal $[\mathrm{A} \ell \, | \, k \mathrm{A}]$- and the $[\mathrm{A} \ell \, | \, k \mathrm{V}]$-connected subgraphs of an arc-colored digraph can be found in polynomial time, respectively.

Note that even though some set $W \subseteq V(G)$ of vertices is the vertex set of an $[\mathrm{A} \ell \, | \, k \mathrm{A}]$- or $[\mathrm{A} \ell \, | \, k \mathrm{V}]$-connected component of an arc-colored digraph $G$, the subgraph $G[W]$ is not necessarily $[\mathrm{A} \ell \, | \, k \mathrm{A}]$- or $[\mathrm{A} \ell \, | \, k \mathrm{V}]$-connected. In the example shown in Figure~\ref{fig:example_AlkA_2}, we observed that $W \colonequals \{ v_1, v_2, v_3, v_4, v_7 \}$ is the vertex set of an $[\mathrm{A} 1 \, | \, 1 \mathrm{V}]$-connected component of $G$. On the other hand, the subgraph $G[W]$ is not $[\mathrm{A} 1 \, | \, 1 \mathrm{V}]$-connected since, after the removal of the green arcs from $G[W]$, the vertex $v_7$ becomes isolated.

In the following, we study such maximal vertex sets that span an arc-color-avoiding connected subgraph. For a set $W \subseteq V(G)$ of vertices, the subgraph $G[W]$ is called a \emph{maximal arc-$\ell$-color-avoiding strongly $k$-arc-connected subgraph}, or for short a \emph{maximal $[\mathrm{A} \ell \, | \, k \mathrm{A}]$-connected subgraph} (and \emph{maximal arc-$\ell$-color-avoiding strongly $k$-vertex-connected subgraph}, or for short \emph{maximal $[\mathrm{A} \ell \, | \, k \mathrm{V}]$-connected subgraph}, respectively), of an arc-colored digraph $G$, if it is inclusion-wise maximal for the property of being $[\mathrm{A} \ell \, | \, k \mathrm{A}]$-connected (or $[\mathrm{A} \ell \, | \, k \mathrm{V}]$-connected).

For instance, the graph shown in Figure~\ref{fig:example_AlkA_2} has only one maximal $[\mathrm{A} 1 \, | \, 1 \mathrm{A}]$-connected subgraph: it is spanned by the vertices $v_1, v_2, v_3, v_4$.

Let us present a polynomial-time algorithm for finding the maximal arc-color-avoiding connected subgraphs of a given arc-colored digraph for any fixed $k, \ell \in \mathbb{Z}_+$. To do so, we use algorithms that find the arc-color-avoiding connected components as subroutines. Let \textbf{$\boldsymbol{[\mathrm{A} \ell \, | \, k \mathrm{A}]}$-Connected\-Components}$(G,f)$ and \textbf{$\boldsymbol{[\mathrm{A} \ell \, | \, k \mathrm{V}]}$-Connected\-Components}$(G,f)$ denote polynomial-time algorithms which respectively return the vertex sets of the $[\mathrm{A} \ell \, | \, k \mathrm{A}]$- and $[\mathrm{A} \ell \, | \, k \mathrm{V}]$-connected components of a digraph $G$ with arc-coloring $f \colon A(G) \to 2^C$, where $C$ is a finite color set. Note that by Theorems~\ref{thm:complexity_of_AlkA} and~\ref{thm:complexity_of_AlkV} such polynomial-time algorithms exist. For ease of notation, let $f_W$ denote the restriction of $f$ to $A \big( G[W] \big)$ for any $W \subseteq V(G)$.

\subsubsection[\texorpdfstring{Maximal $[\mathrm{A} \ell \, | \, k \mathrm{A}]$}{}-connected subgraphs]{\texorpdfstring{Maximal $\boldsymbol{[\mathrm{A} \ell \, | \, k \mathrm{A}]}$}{}-connected subgraphs}\label{section:AlkA_new_components}

Now we consider the problem of finding the maximal $[\mathrm{A} \ell \, | \, k \mathrm{A}]$-connected subgraphs of an arc-colored digraph. Algorithm~\ref{algo:AlkA_components_2} iteratively takes the $[\mathrm{A} \ell \, | \, k \mathrm{A}]$-connected components of the $[\mathrm{A} \ell \, | \, k \mathrm{A}]$-connected components. At the end, we examine each 1-vertex $[\mathrm{A} \ell \, | \, k \mathrm{A}]$-connected component whether at least $k$ loops remain in it after deleting any $\ell$ colors; if not, then we discard it. The remaining components are precisely the maximal $[\mathrm{A} \ell \, | \, k \mathrm{A}]$-connected subgraphs.

\bigskip

\begin{algorithm}[H] 
 \caption{Algorithm for finding the maximal $[\mathrm{A} \ell \, | \, k \mathrm{A}]$-connected subgraphs}
 \label{algo:AlkA_components_2}
 \DontPrintSemicolon
 
 \KwIn{a digraph $G$, a finite color set $C$, and a function $f \colon A(G) \to 2^C$.}
 \KwOut{the vertex sets of the maximal $[\mathrm{A} \ell \, | \, k \mathrm{A}]$-connected subgraphs of $G$.}
 
 \smallskip

 $\mathcal{V}_{-1} \gets \{ \emptyset \}$\;
 $\mathcal{V}_0 \gets \big\{ V(G) \big\}$\;
 $i \gets 0$\;
 \While{$\mathcal{V}_i \ne \mathcal{V}_{i-1}$}{
     $\mathcal{V}_{i+1} \gets \emptyset$\;
     \For{$W \in \mathcal{V}_i$}{
         $\mathcal{V}_{i+1} \gets \mathcal{V}_{i+1} \cup [\mathrm{A} \ell \, | \, k \mathrm{A}] \textrm{-ConnectedComponents} \big( G[W],f_W \big)$\; 
     }
     $i \gets i+1$\;
 }
 \For{$W \in \mathcal{V}_i$}{
     \If{$|W| = 1$ \textbf{and} $G[W]$ is not $[\mathrm{A} \ell \, | \, k \mathrm{A}]$-connected}{
         $\mathcal{V}_i \gets \mathcal{V}_i \setminus \{ W \}$\;
     }
 }
 
\Return{$\mathcal{V}_i$}\;
\end{algorithm}

\bigskip

Now we show the correctness of Algorithm~\ref{algo:AlkA_components_2} and we analyze its running time.

\begin{theorem} \label{thm:AlkA_2_algo}
Let $k, \ell \in \mathbb{Z}_+$ be fixed. Given a digraph $G$, a finite color set $C$ and a function $f \colon A(G) \to 2^C$, Algorithm~\ref{algo:AlkA_components_2} outputs the vertex sets of the maximal $[\mathrm{A} \ell \, | \, k \mathrm{A}]$-connected subgraphs of $G$ and runs in polynomial time.
\end{theorem}
\begin{proof}
 We discuss the time complexity and the correctness of the algorithm separately.

 \smallskip
 
 \noindent \textit{Time complexity.} Observe that since the $[\mathrm{A} \ell \, | \, k \mathrm{A}]$-connected components form a partition of the vertex set of $G$, so does $ \mathcal{V}_i$ for each $i \in \mathbb{Z}_+ \cup \{ 0 \}$. Therefore, every for loop takes at most $n$ iterations, and each iteration invokes the subroutine $[\mathrm{A} \ell \, | \, k \mathrm{A}]$-Connected\-Components on a subgraph of $G$.
 
 Now we show that the algorithm terminates after at most $n$ iterations of the while loop. To see this, let $i \in \mathbb{Z}_+$ and $W \in \mathcal{V}_i$. Since $\mathcal{V}_i$ is a partition of the vertex set, there exists a unique $U \in \mathcal{V}_{i-1}$ such that $W \subseteq U$. We distinguish two cases: $W = U$ or $W \subsetneq U$. In the first case, $G[U]$ is $[\mathrm{A} \ell \, | \, k \mathrm{A}]$-connected or $|U| = 1$. (If $|U| = 1$ but $G[U]$ is not $[\mathrm{A} \ell \, | \, k \mathrm{A}]$-connected, then $U$ is removed from the output set at the last phase of the algorithm.) Thus, $U$ is, but none of its proper subsets are contained in $\mathcal{V}_j$ for any $j \ge i$. In the latter case, clearly $|W| < |U|$ holds. Furthermore, the only element of $\mathcal{V}_0$ has size $n$ and since $W \notin \mathcal{V}_{i-1}$, we obtain $|W| \le n-i$ by the previous observation. Also note that $\emptyset \notin \mathcal{V}_j$ for all $j \in \mathbb{Z}_+$. Hence, the while loop indeed terminates after at most $n$ iterations.

 Therefore, the subroutine $[\mathrm{A} \ell \, | \, k \mathrm{A}]$-Connected\-Components on a subgraph of $G$ is invoked at most $n^2$ times, and thus by Theorem~\ref{thm:complexity_of_AlkA}, the algorithm runs in polynomial time.

 \smallskip
 
 \noindent \textit{Correctness.} Clearly, the output $\mathcal{V}^*$ of the algorithm is a family of the vertex sets of $[\mathrm{A} \ell \, | \, k \mathrm{A}]$-connected subgraphs of $G$. We only need to show that each of these subgraphs is maximal and that all the maximal $[\mathrm{A} \ell \, | \, k \mathrm{A}]$-connected subgraphs are contained in $\mathcal{V}^*$.
 
 Let us start with proving the inclusion-wise maximality of the members of $\mathcal{V}^*$. Let $X^* \in \mathcal{V}^*$, and let $X \subseteq V(G)$ be the vertex set of a maximal $[\mathrm{A} \ell \, | \, k \mathrm{A}]$-connected subgraph of $G$ such that $X^* \subseteq X$. Note that $X$ is unique. Now we show that $X^* = X$. Suppose to the contrary that $X^* \ne X$. Then, since $\mathcal{V}_i$ is a partition of $V(G)$ for any $i \in \mathbb{Z}_+ \cup \{ 0 \}$ and $\mathcal{V}_0 = \big\{ V(G) \big\}$, there exist an index $j \in \mathbb{Z}_+$ and sets $U \in \mathcal{V}_{j-1}$ and $W \in \mathcal{V}_j$ such that $X^* \subsetneq X \subseteq U$ and $X^* \subseteq W$ but $X \not\subseteq W$. Thus $X$ does not span a $[\mathrm{A} \ell \, | \, k \mathrm{A}]$-connected subgraph in $G[U]$, which implies that it also does not span a $[\mathrm{A} \ell \, | \, k \mathrm{A}]$-connected subgraph of $G$, a contradiction.
 
 Finally, we show that the vertex set of any maximal $[\mathrm{A} \ell \, | \, k \mathrm{A}]$-connected subgraph of $G$ is contained in $\mathcal{V}^*$. Let $X \subseteq V(G)$ be the vertex set of a maximal $[\mathrm{A} \ell \, | \, k \mathrm{A}]$-connected subgraph of $G$. By the above and since $X \subseteq V(G) \in \mathcal{V}_0$ and $\mathcal{V}_i$ forms a partition of $V(G)$ for each $i \in \mathbb{Z}_+$, there exists $X^* \in \mathcal{V}^*$ such that $X \subseteq X^*$. Also by the above, $X^*$ spans a maximal $[\mathrm{A} \ell \, | \, k \mathrm{A}]$-connected subgraph in $G$. Thus, we have $X^* = X$, so $X \in \mathcal{V}^*$.
\end{proof}

By Theorem~\ref{thm:AlkA_2_algo}, we conclude the following.

\begin{corollary} \label{cor:complexity_of_AlkA_2}
 For any $k, \ell \in \mathbb{Z}_+$, the following problem can be solved in polynomial time.

 \smallskip
 
 \noindent {\bf \scshape \textsc{Maximal-$\boldsymbol{[\mathrm{A} \ell \, | \, k \mathrm{A}]}$-ConnectedSubgraphs}}
 \textit{Instance:} a digraph $G$, a finite color set $C$, a function $f \colon A(G) \to 2^C$, and $m \in \mathbb{Z}_+$. \\
 \textit{Question:} is it true that $G$ has a maximal $[\mathrm{A} \ell \, | \, k \mathrm{A}]$-connected subgraph on at least $m$ vertices?
 
 \smallskip
 
 \noindent Moreover, all the maximal $[\mathrm{A} \ell \, | \, k \mathrm{A}]$-connected subgraphs of $G$ can be found in polynomial time.
\end{corollary}

\subsubsection[\texorpdfstring{Maximal $[\mathrm{A} \ell \, | \, k \mathrm{V}]$}{}-connected subgraphs]{\texorpdfstring{Maximal $\boldsymbol{[\mathrm{A} \ell \, | \, k \mathrm{V}]}$}{}-connected subgraphs}
\label{section:AlkV_new_components}

Now we consider the problem of finding the maximal $[\mathrm{A} \ell \, | \, k \mathrm{V}]$-connected subgraphs of an arc-colored digraph. Since $[\mathrm{A} \ell \, | \, k \mathrm{V}]$-connectivity is not transitive, we make use of the following lemma instead.

\begin{lemma} \label{lemma:AlkV_2_intersection}
 Let $k, \ell \in \mathbb{Z}_+$, let $G$ be a digraph, let $C$ be a finite color set, and let $f \colon A(G) \to 2^C$. Then, any two distinct $[\mathrm{A} \ell \, | \, k \mathrm{V}]$-connected components of $G$ has at most $k-1$ common vertices.
\end{lemma}
\begin{proof}
 Let $V_1$ and $V_2$ be the vertex sets of two $[\mathrm{A} \ell \, | \, k \mathrm{V}]$-connected components such that there exist $k$ distinct vertices $v_1, \ldots, v_k \in V_1 \cap V_2$. We show that any two vertices of $V_1 \cup V_2$ are $[\mathrm{A} \ell \, | \, k \mathrm{V}]$-connected, thus by the inclusion-wise maximality of $V_1$ and $V_2$, we have $V_1 = V_2$.

 Suppose to the contrary that there exist two vertices $u, w \in V_1 \cup V_2$, subsets $C' \subseteq C$ and $A' \subseteq A(G)$ and $V' \subseteq V(G) \setminus \{ u, w \}$ with $|C'| \le \ell$ and $|A'| + |V'| \le k-1$ such that after the removal of $C'$, $A'$, and $V'$ from $G$ there does not remain any $u$-$w$ or any $w$-$u$ dipath. Clearly, $\{ v_1, \ldots, v_k \} \setminus V' \ne \emptyset$; let $i \in [k]$ such that $v_i \notin V'$. Without loss of generality, we can assume that $u \in V_1$. Since $V_1$ is the vertex set of an $[\mathrm{A} \ell \, | \, k \mathrm{V}]$-connected component, after the removal of $C'$, $A'$, and $V'$ from $G$, there remains a $u$-$v_i$ and a $v_i$-$u$ dipath. Analogously, the same holds for $v_i$ and $w$. Therefore, after the removal of $C'$, $A'$, and $V'$ from $G$, there remains a directed $u$-$w$ and a directed $w$-$u$ walk via $v_i$, which is a contradiction.
\end{proof}

Let us present our algorithm, which considers all sets of at most $k$ vertices as follows. For every set of size exactly $k$, we take the $[\mathrm{A} \ell \, | \, k \mathrm{V}]$-connected components of $G$. If none of these components contains our set entirely, we proceed to the next set. Otherwise, by Lemma~\ref{lemma:AlkV_2_intersection}, there is exactly one such component; we then take the $[\mathrm{A} \ell \, | \, k \mathrm{V}]$-connected components of this component and continue this procedure until we obtain the maximal $[\mathrm{A} \ell \, | \, k \mathrm{V}]$-connected subgraphs of size at least $k$. We then consider all sets of at most $k$ vertices in decreasing order of size. For each such set, if it is not contained in any of the previously determined maximal $[\mathrm{A} \ell \, | \, k \mathrm{V}]$-connected subgraphs, we check whether it spans an $[\mathrm{A} \ell \, | \, k \mathrm{V}]$-connected subgraph. The so-obtained subgraphs are precisely the maximal $[\mathrm{A} \ell \, | \, k \mathrm{V}]$-connected ones.

\bigskip

\begin{algorithm}[H] 
 \caption{Algorithm for finding the maximal $[\mathrm{A} \ell \, | \, k \mathrm{V}]$-connected subgraphs}
 \label{algo:AlkV_components_2}
 \DontPrintSemicolon
 
 \KwIn{a digraph $G$, a finite color set $C$, and a function $f \colon A(G) \to 2^C$.}
 \KwOut{the vertex sets of the maximal $[\mathrm{A} \ell \, | \, k \mathrm{V}]$-connected subgraphs of $G$.}
 
 \smallskip

 $\mathcal{V}^* \gets \emptyset$\;
 \For{$W \subseteq V(G)$ with $|W| = k$}{
     $\mathcal{V} \gets [\mathrm{A} \ell \, | \, k \mathrm{V}] \textrm{-ConnectedComponents} \big( G, f \big)$\;
     \While{there exists $U \in \mathcal{V}$ for which $W \subseteq U$}{
         \If{$G[U]$ is $[\mathrm{A} \ell \, | \, k \mathrm{V}]$-connected}{
             $\mathcal{V}^* \gets \mathcal{V}^* \cup \{ U \}$\;
             \textbf{break}\;
         }
         \ElseIf{$U = W$}{
             \textbf{break}\;
         }
         \Else{
             $\mathcal{V} \gets [\mathrm{A} \ell \, | \, k \mathrm{V}] \textrm{-ConnectedComponents} \big( G[U], f_U \big)$
         }
     }
 }
 \For{$i = k-1$ \textbf{to} 1}{
     \For{$W \subseteq V(G)$ with $|W| = i$}{
         \If{$G[W]$ is $[\mathrm{A} \ell \, | \, k \mathrm{V}]$-connected \textbf{and} there does not exist $U \in \mathcal{V}^*$ with $W \subset U$}{
             $\mathcal{V}^* \gets \mathcal{V}^* \cup \{ W \}$\;
         }
     }
 }
\Return{$\mathcal{V}^*$}\;
\end{algorithm}

\bigskip

Let us show the correctness of Algorithm~\ref{algo:AlkV_components_2} and analyze its running time.

\begin{theorem} \label{thm:AlkV_2_algo}
 Let $k, \ell \in \mathbb{Z}_+$ be fixed. Given a digraph $G$, a finite color set $C$ and a function $f \colon A(G) \to 2^C$, Algorithm~\ref{algo:AlkV_components_2} outputs the vertex sets of the maximal $[\mathrm{A} \ell \, | \, k \mathrm{V}]$-connected subgraphs of $G$ and runs in polynomial time.
\end{theorem}

\begin{proof}
 We discuss the correctness and the time complexity of the algorithm separately.

 \smallskip
 
 \noindent \textit{Correctness.} The correctness of the algorithm directly follows from Lemma~\ref{lemma:AlkV_2_intersection}.

 \smallskip
 
 \noindent \textit{Time complexity.} It is not difficult to see that the for loops altogether take $\sum_{i=1}^k \binom{n}{i}$ iterations, and each iteration invokes the subroutine $[\mathrm{A} \ell \, | \, k \mathrm{V}]$-Connected\-Components on some subgraph of $G$ at most $n$ times.
 
 Therefore, the subroutine $[\mathrm{A} \ell \, | \, k \mathrm{V}]$-Connected\-Components on some subgraph of $G$ is invoked $O(n^{k+1})$ times, and by Theorem~\ref{thm:complexity_of_AlkV}, this subroutine has polynomial running time.
 
 Hence, Algorithm~\ref{algo:AlkV_components_2} runs in polynomial time.
\end{proof}

By Theorem~\ref{thm:AlkV_2_algo}, we conclude that for any fixed $k, \ell \in \mathbb{Z}_+$, the maximal $[\mathrm{A} \ell \, | \, k \mathrm{V}]$-connected subgraphs can be found in polynomial time.

\begin{corollary} \label{cor:complexity_of_AlkV_2}
 For any $k, \ell \in \mathbb{Z}_+$, the following problem can be solved in polynomial time.

 \smallskip
 
 \noindent {\bf \scshape \textsc{Maximal-$\boldsymbol{[\mathrm{A} \ell \, | \, k \mathrm{V}]}$-ConnectedSubgraphs}} \\
 \textit{Instance:} a digraph $G$, a finite color set $C$, a function $f \colon A(G) \to 2^C$, and $m \in \mathbb{Z}_+$. \\
 \textit{Question:} is it true that $G$ has a maximal $[\mathrm{A} \ell \, | \, k \mathrm{V}]$-connected subgraph on at least $m$ vertices?
 
 \smallskip
 
 \noindent Moreover, all the maximal $[\mathrm{A} \ell \, | \, k \mathrm{V}]$-connected subgraphs of $G$ can be found in polynomial time.
\end{corollary}

\section{Vertex-color-avoiding strong connectivity} \label{section:Vlk}

In this section, we focus on vertex-color-avoiding strong connectivity. First, we present the formal definitions. Then, in Section~\ref{section:Vlk_old_components}, we show that in general the problem of finding the vertex-color-avoiding strongly connected components of a vertex-colored digraph is NP-hard, but when every vertex is single-colored, then the problem can be solved in polynomial time. Finally, in Section~\ref{section:Vlk_new_components}, we show that the same holds for the problem of finding the maximal vertex-color-avoiding strongly connected subgraphs.

\begin{definition}[${[\mathrm{V} \ell \, | \, k \mathrm{A}]}$- and ${[\mathrm{V} \ell \, | \, k \mathrm{V}]}$-connectivity] \label{def:VlkA_and_VlkV}
 Let $k, \ell \in \mathbb{Z}_+$, let $G$ be a digraph, let $C$ be a finite color set, and let $f \colon V(G) \to 2^C$ be a function assigning some colors to every vertex. Two distinct vertices $u, v \in V(G)$ are called \emph{vertex-$\ell$-color-avoiding strongly $k$-arc-connected}, or for short \emph{$[\mathrm{V} \ell \, | \, k \mathrm{A}]$-connected} (and respectively, \emph{vertex-$\ell$-color-avoiding strongly $k$-vertex-connected}, or for short \emph{$[\mathrm{V} \ell \, | \, k \mathrm{V}]$-connected}) if after the removal of any at most $\ell$ colors, either at least one of the vertices $u$ and $v$ is removed, or they remain strongly $k$-arc-connected (or strongly $k$-vertex-connected).
\end{definition}

In Figure~\ref{fig:example_VlkA}, we can see an example for Definition~\ref{def:VlkA_and_VlkV}. Considering each edge of an undirected graph as two oppositely directed arcs, the concept of vertex-color-avoiding strong connectivity can also be applied to undirected graphs.

\begin{figure}[H]
\centering
\begin{tikzpicture}[scale=1]
 \tikzstyle{vertex_red}=[draw, thick, red!50!black, shape=rectangle, fill=red!50, minimum size=8pt, inner sep=0]
 \tikzstyle{vertex_blue}=[draw, thick, blue!50!black, shape=regular polygon, regular polygon sides=3, fill=blue!50, minimum size=12pt, inner sep=0]
 \tikzstyle{vertex_green}=[draw, green!50!black, kite, fill=green!50, minimum size=11pt, inner sep=0.1pt]
 \tikzstyle{vertex_yellow}=[draw, yellow!50!black, shape=regular polygon, regular polygon sides=5, fill=yellow!50, minimum size=11pt, inner sep=0]

 \node at (-1.125,1) {$G_1$};
 \node[matrix, draw, shape=circle, thick, inner sep=2pt] (v1) at (90:1) [label=above:$v_1$]
 {\node[vertex_red, minimum size=6pt] {}; &  \\
   & \node[vertex_blue, minimum size=9pt] {}; \\ };
 \node[vertex_red] (v2) at (210:1) [label=below left:$v_2$] {};
 \node[vertex_blue] (v3) at (330:1) [label=below right:$v_3$] {};
 
 \draw[very thick,->,>={LaTeX[black,length=8pt, width=8pt]}] (v1) -- (v2);
 \draw[very thick,->,>={LaTeX[black,length=8pt, width=8pt]}] (v2) -- (v3);
 \draw[very thick,->,>={LaTeX[black,length=8pt, width=8pt]}] (v3) -- (v1);
 
 \begin{scope}[shift={(5,0)}]
 \node at (-1.125,1) {$G_2$};
 
 \node[vertex_green] (w1) at (90:1) [label=above:$w_1$] {};
 \node[vertex_red] (w2) at (210:1) [label=below left:$w_2$] {};
 \node[vertex_blue] (w3) at (330:1) [label=below right:$w_3$] {};
 
 \draw[very thick,->,>={LaTeX[black,length=8pt, width=8pt]}] (w1) -- (w2);
 \draw[very thick,->,>={LaTeX[black,length=8pt, width=8pt]}] (w2) -- (w3);
 \draw[very thick,->,>={LaTeX[black,length=8pt, width=8pt]}] (w3) -- (w1);
 \end{scope}
\end{tikzpicture}
\caption{Examples for $[\mathrm{V} \ell \, | \, k \mathrm{A}]$-connected vertices. The colors red, blue, and green are denoted by (red) squares, (blue) triangles, and (green) deltoids, respectively. The vertex $v_1$ is assigned two colors, namely red and blue. In the digraph $G_1$, any two vertices are $[\mathrm{V} 2 \, | \, 1 \mathrm{A}]$-connected; thus, $G_1$ is $[\mathrm{V} 2 \, | \, 1 \mathrm{A}]$-connected. In the digraph $G_2$, no pair of vertices are $[\mathrm{V} 1 \, | \, 1 \mathrm{A}]$-connected.}
\label{fig:example_VlkA}
\end{figure}
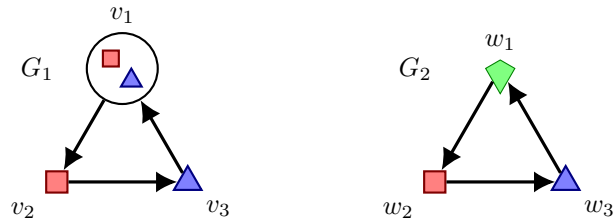

It follows directly from Menger's classical theorems~\cite{menger1927allgemeinen} that
 \begin{enumerate}[topsep=1.5pt, itemsep=-1.5pt, label=(\alph*)] \setcounter{enumi}{2}
  \item \label{item:VlkA} \linkdest{item:VlkA_of_lemma}{} the vertices $u$ and $v$ of a vertex-colored digraph $G$ are not $[\mathrm{V} \ell \mid k \mathrm{A}]$-connected if and only if there exist $C' \subseteq C \setminus \big( f(u) \cup f(v) \big)$ and $A' \subseteq A(G)$ with $|C'| \le \ell$ and $|A'| \le k-1$ whose removal from $G$ does not leave any $u$-$v$ or any $v$-$u$ dipath;
  
  \item \label{item:VlkV} \linkdest{item:VlkV_of_lemma}{} the vertices $u$ and $v$ of a vertex-colored digraph $G$ are not $[\mathrm{V} \ell \, | \, k \mathrm{V}]$-connected if and only if there exist $C' \subseteq C \setminus \big( f(u) \cup f(v) \big)$ and $A' \subseteq A(G)$ and $V' \subseteq V(G) \setminus \{ u, v \}$ with $|C'| \le \ell$ and $|A'| + |V'| \le k-1$ whose removal from $G$ does not leave any $u$-$v$ or any $v$-$u$ dipath.
 \end{enumerate}

\medskip

Also observe that the notion of $[\mathrm{V} \ell \, | \, 1 \mathrm{A}]$- and $[\mathrm{V} \ell \, | \, 1 \mathrm{V}]$-connectivity is the same. Finally, note that if every vertex of a digraph is colored with zero colors, then two vertices are $[\mathrm{V} \ell \, | \, k \mathrm{A}]$-connected (and respectively, $[\mathrm{V} \ell \, | \, k \mathrm{V}]$-connected) if and only if they are strongly $k$-arc-connected (or strongly $k$-vertex-connected).

\subsection{Vertex-color-avoiding strongly connected components} \label{section:Vlk_old_components}

In this section, we define vertex-color-avoiding strongly connected components and prove that finding the $[\mathrm{V} \ell \, | \, k \mathrm{A}]$- and the $[\mathrm{V} \ell \, | \, k \mathrm{V}]$-connected components of a vertex-colored digraph is NP-hard. On the other hand, in Sections~\ref{section:VlkA_old_components} and~\ref{section:VlkV_old_components}, we show that the same problem in the setting where every vertex is single-colored can be solved in polynomial time.

It is not difficult to see that the relations of $[\mathrm{V} \ell \, | \, k \mathrm{A}]$- and $[\mathrm{V} \ell \, | \, k \mathrm{V}]$-connectivity are not transitive, and thereby they are not equivalence relations for any $k, \ell \in \mathbb{Z}_+$. 

For any $k, \ell \in \mathbb{Z}_+$, the \emph{vertex-$\ell$-color-avoiding strongly $k$-arc-connected components}, or for short the \emph{$[\mathrm{V} \ell \, | \, k \mathrm{A}]$-connected components} (and respectively, \emph{vertex-$\ell$-color-avoiding strongly $k$-vertex-connected components}, or for short the \emph{$[\mathrm{V} \ell \, | \, k \mathrm{V}]$-connected components}) of a vertex-colored digraph $G$ are maximal sets of vertices such that any two of these vertices are $[\mathrm{V} \ell \, | \, k \mathrm{A}]$-connected (or $[\mathrm{V} \ell \, | \, k \mathrm{V}]$-connected) in $G$, respectively. If a vertex $v \in V(G)$ is not $[\mathrm{V} \ell \, | \, k \mathrm{A}]$- or $[\mathrm{V} \ell \, | \, k \mathrm{V}]$-connected to any other vertex of the digraph $G$ for some $k, \ell \in \mathbb{Z}_+$, then $v$ is considered to form a $[\mathrm{V} \ell \, | \, k \mathrm{A}]$- or $[\mathrm{V} \ell \, | \, k \mathrm{V}]$-connected component of $G$ in itself, respectively.

We say that a vertex-colored digraph on at least two vertices is \emph{vertex-$\ell$-color-avoiding strongly $k$-arc-connected}, or for short $[\mathrm{V} \ell \, | \, k \mathrm{A}]$-connected (and respectively, \emph{vertex-$\ell$-color-avoiding strongly $k$-vertex-connected}, or for short $[\mathrm{V} \ell \, | \, k \mathrm{V}]$-connected) if it has exactly one $[\mathrm{V} \ell \, | \, k \mathrm{A}]$-connected component (or $[\mathrm{V} \ell \, | \, k \mathrm{V}]$-connected component). A vertex-colored digraph on exactly one vertex is $[\mathrm{V} \ell \, | \, k \mathrm{A}]$-connected (and respectively, $[\mathrm{V} \ell \, | \, k \mathrm{V}]$-connected) if it has at least $k$ loops. For an example, see Figure~\ref{fig:example_VlkA}.

To find the vertex-color-avoiding strongly connected components of a vertex-colored digraph, we use auxiliary graphs.

\begin{definition}[Auxiliary graphs $G_{[\mathrm{V} \ell \mid k \mathrm{A}]}$ and $G_{[\mathrm{V} \ell \mid k \mathrm{V}]}$]
 Let $k, \ell \in \mathbb{Z}_+$, let $G$ be a digraph, let $C$ be a finite color set, and let $f \colon V(G) \to 2^C$. The \emph{auxiliary graphs} of $G$ are the undirected graphs $G_{[\mathrm{V} \ell \mid k \mathrm{A}]}$ and $G_{[\mathrm{V} \ell \mid k \mathrm{V}]}$ on the same vertex set as $G$ and two vertices of theirs are adjacent if and only if the corresponding vertices in $G$ are $[\mathrm{V} \ell \, | \, k \mathrm{A}]$- or $[\mathrm{V} \ell \, | \, k \mathrm{V}]$-connected, respectively.
\end{definition}

Similarly as in the case of $[\mathrm{A} \ell \, | \, k \mathrm{A}]$- and $[\mathrm{A} \ell \, | \, k \mathrm{V}]$-connectivity, we have that $W \subseteq V(G) $ is the vertex set of a $[\mathrm{V} \ell \, | \, k \mathrm{A}]$- or a $[\mathrm{V} \ell \, | \, k \mathrm{V}]$-connected component of $G$ if and only if $W$ spans a maximal clique in $G_{[\mathrm{V} \ell \mid k \mathrm{A}]}$ or $G_{[\mathrm{V} \ell \mid k \mathrm{V}]}$, respectively.

Also, observe that for fixed $k, \ell \in \mathbb{Z}_+$, we can construct the auxiliary graphs $G_{[\mathrm{V} \ell \mid k \mathrm{A}]}$ and $G_{[\mathrm{V} \ell \mid k \mathrm{V}]}$ in polynomial time by checking for every pair of vertices whether for each $C' \subseteq C$ with $|C'| \le \ell$ (including $C'=\emptyset$) after the removal of $C'$ from $G$ either at least one of the two vertices is removed or there remain at least $k$ pairwise arc- or internally vertex-disjoint dipaths between them in both directions. This can be done by running a maximum $s$-$t$ flow algorithm for any $s,t \in V(G)$ and for any subset $C' \subseteq C \setminus \big( f(s) \cup f(t) \big)$ with $s \ne t$ and $|C'| \le \ell$.

Now we prove that unlike finding arc-color-avoiding strongly connected components, the problem of finding vertex-color-avoiding strongly connected components is NP-hard.

\begin{theorem} \label{thm:complexity_of_VlkA_and_VlkV}
For any $k, \ell \in \mathbb{Z}_+$, the following problems are NP-complete even for undirected graphs.

 \smallskip
 
 \noindent {\bf \textsc{$\boldsymbol{[\mathrm{V} \ell \, | \, k \mathrm{A}]}$-ConnectedComponents}}

 \smallskip
 
 \noindent \textit{Instance:} a digraph $G$, a finite color set $C$, a function $f \colon V(G) \to 2^C$, and $m \in \mathbb{Z}_+$.
 
 \noindent \textit{Question:} is it true that $G$ has a $[\mathrm{V} \ell \, | \, k \mathrm{A}]$-connected component of size at least~$m$?

 \medskip
 
 \noindent {\bf \textsc{$\boldsymbol{[\mathrm{V} \ell \, | \, k \mathrm{V}]}$-ConnectedComponents}} 

 \smallskip

 \noindent \textit{Instance:} a digraph $G$, a finite color set $C$, a function $f \colon V(G) \to 2^C$, and $m \in \mathbb{Z}_+$.
 
 \noindent \textit{Question:} is it true that $G$ has a $[\mathrm{V} \ell \, | \, k \mathrm{V}]$-connected component of size at least~$m$?
\end{theorem}
\begin{proof}
 Let $k, \ell \in \mathbb{Z}_+$ be fixed. Obviously, these problems are in NP.
 
 Now we show that these problems are NP-hard by reducing the NP-complete problem \textsc{Clique}~\cite{karp1972reducibility} to them. Let $G$ be an arbitrary undirected graph. Let $G'$ be the undirected graph that can be obtained from $G$ by replacing each of its edges with $k$ parallel edges. Assign a color to any two vertices of $G$, and add this color to the set of colors of every other vertex in $G'$; so altogether we use $\binom{n}{2}$ colors and every vertex has $\binom{n}{2} - (n-1)$ colors in its set of colors. For an example, see Figure~\ref{fig:example_VlkA_and_VlkV_complexity_construction}.
 
 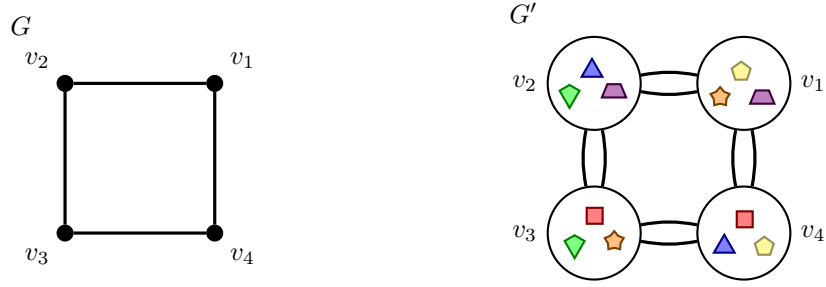
\begin{figure}[H]
 \centering
 \begin{tikzpicture}[scale=1.4]
  \tikzstyle{vertex}=[draw,circle,fill,minimum size=6,inner sep=0]
 
  \tikzstyle{vertex_red}=[draw, thick, red!50!black, shape=rectangle, fill=red!50, minimum size=8pt, inner sep=0]
  \tikzstyle{vertex_blue}=[draw, thick, blue!50!black, shape=regular polygon, regular polygon sides=3, fill=blue!50, minimum size=12pt, inner sep=0]
  \tikzstyle{vertex_green}=[draw, green!50!black, kite, fill=green!50, minimum size=11pt, inner sep=0.1pt]
  \tikzstyle{vertex_yellow}=[draw, yellow!50!black, shape=regular polygon, regular polygon sides=5, fill=yellow!50, minimum size=11pt, inner sep=0]
  \tikzstyle{vertex_orange}=[draw, orange!50!black, shape=star, star points=5, fill=orange!50, minimum size=11pt, inner sep=0]
  \tikzstyle{vertex_purple}=[draw, violet!50!black, shape=trapezium, trapezium angle=67.5, fill=violet!50, minimum size=7pt, inner sep=0.1pt]
 
  \node at (-1.125,1.25) {$G$};
  
  \node[vertex] (v1) at (45:1) [label=above right:$v_1$] {};
  \node[vertex] (v2) at (135:1) [label=above left:$v_2$] {};
  \node[vertex] (v3) at (225:1) [label=below left:$v_3$] {};
  \node[vertex] (v4) at (315:1) [label=below right:$v_4$] {};
  
  \draw[very thick] (v1) -- (v2) -- (v3) -- (v4) -- (v1);
  
  \begin{scope}[shift={(5,0)}]
  \node at (-1.375,1.375) {$G'$};
 
  \node[matrix, draw, shape=circle, thick, minimum size=35pt, inner sep=0pt] (w1) at (45:1) [label=right:$v_1$]
  { & \node[vertex_yellow, minimum size=7.5pt] {}; & \\[2pt]
   \node[vertex_orange, minimum size=7.5pt] {}; &  & \hspace{-1pt} \node[vertex_purple, minimum size=5pt] {}; \\ };

  \node[matrix, draw, shape=circle, thick, minimum size=35pt, inner sep=0pt] (w2) at (135:1) [label=left:$v_2$]
  { & \node[vertex_blue, minimum size=9pt] {}; & \\[2pt]
   \node[vertex_green, minimum size=7.5pt] {}; &  & \hspace{-1pt} \node[vertex_purple, minimum size=5pt] {}; \\ };

  \node[matrix, draw, shape=circle, thick, minimum size=35pt, inner sep=0pt] (w3) at (225:1) [label=left:$v_3$]
  { & \node[vertex_red, minimum size=6pt] {}; & \\[2pt]
   \node[vertex_green, minimum size=7.5pt] {}; &  & \node[vertex_orange, minimum size=7.5pt] {}; \\ };

  \node[matrix, draw, shape=circle, thick, minimum size=35pt, inner sep=0pt] (w4) at (315:1) [label=right:$v_4$]
  { & \node[vertex_red, minimum size=6pt] {}; & \\[2pt]
   \node[vertex_blue, minimum size=9pt] {}; &  & \node[vertex_yellow, minimum size=7.5pt] {}; \\ };

  \draw[very thick] (w1) to [bend left=10] (w2);
  \draw[very thick] (w1) to [bend right=10] (w2);

  \draw[very thick] (w2) to [bend left=10] (w3);
  \draw[very thick] (w2) to [bend right=10] (w3);

  \draw[very thick] (w3) to [bend left=10] (w4);
  \draw[very thick] (w3) to [bend right=10] (w4);

  \draw[very thick] (w4) to [bend left=10] (w1);
  \draw[very thick] (w4) to [bend right=10] (w1);
  \end{scope}
 \end{tikzpicture}
 \caption{The vertex-colored undirected graph $G'$ constructed in the proof of Theorem~\ref{thm:complexity_of_VlkA_and_VlkV} to $G$ in the case $k = 2$ (and with $\ell \in \mathbb{Z}_+$ being arbitrary). We assign the color red (denoted by squares) to the vertices $v_1$ and $v_2$; thus, this color appears in the sets of colors of $v'_3$ and $v'_4$ in $G'$. Blue (denoted by triangles) is assigned to the vertices $v_1$ and $v_3$, green (denoted by deltoids) to the vertices $v_1$ and $v_4$, yellow (denoted by pentagons) to the vertices $v_2$ and $v_3$, and orange (denoted by stars) to the vertices $v_2$ and $v_4$, and purple (denoted by trapezoids) to the vertices $v_3$ and $v_4$. Then, for example, the set of colors of $v'_1$ consists of yellow, orange, and purple.}
 \label{fig:example_VlkA_and_VlkV_complexity_construction}
 \end{figure}

 It is not difficult to see that $G'_{[\mathrm{V} \ell \mid k \mathrm{A}]} = G'_{[\mathrm{V} \ell \mid k \mathrm{V}]} = G$, thus the following are equivalent:
 \begin{itemize}[topsep=3pt, itemsep=-1.5pt]
  \item[--] $G$ has a clique of size at least~$m$,
  \item[--] $G'$ has a $[\mathrm{V} \ell \, | \, k \mathrm{A}]$-connected component of size at least~$m$, and
  \item[--] $G'$ has a $[\mathrm{V} \ell \, | \, k \mathrm{V}]$-connected component of size at least~$m$.
 \end{itemize}
 Therefore, the problems \textsc{$[\mathrm{V} \ell \, | \, k \mathrm{A}]$-} and \textsc{$[\mathrm{V} \ell \, | \, k \mathrm{V}]$-ConnectedComponents} are indeed NP-complete.
\end{proof}

Note that in the proof of Theorem~\ref{thm:complexity_of_VlkA_and_VlkV}, the size of the sets of colors is a polynomial in $n$. An interesting question is the case when the sets of colors of the vertices have constant sizes. In Sections~\ref{section:VlkA_old_components} and~\ref{section:VlkV_old_components}, we consider the case when every vertex has exactly one color in its set of colors. The case when the sets of colors of the vertices have arbitrary constant sizes remains open.

\subsubsection[\texorpdfstring{The $[\mathrm{V} \ell \, | \, k \mathrm{A}]$}{}-connected components when every vertex is single-colored]{The \texorpdfstring{$\boldsymbol{[\mathrm{V} \ell \, | \, k \mathrm{A}]}$}{}-connected components when every vertex is single-colored} \label{section:VlkA_old_components}

Even though generally the auxiliary graph $G_{[\mathrm{V} \ell \mid k \mathrm{A}]}$ can be arbitrary (see the proof of Theorem~\ref{thm:complexity_of_VlkA_and_VlkV}), the following lemma describes some special structure of $G_{[\mathrm{V} \ell \mid k \mathrm{A}]}$ if every vertex is single-colored. Since the proof uses similar ideas to that of Lemma~\ref{lemma:property_of_AlkV}, we defer it to the \hyperref[proof:property_of_VlkA_single]{Appendix}.

\begin{lemma} \label{lemma:property_of_VlkA_single}
 Let $k, \ell \in \mathbb{Z}_+$, let $G$ be a digraph, let $C$ be a finite color set, and let $f: \ V(G) \to C$. Then the auxiliary graph $G_{[\mathrm{V} \ell \mid k \mathrm{A}]}$ is $\overline{(\ell + 2) K_2}$-free.
\end{lemma}

By Theorem~\ref{thm:pK2_free_listing_max_clique}, we conclude that for any fixed $k, \ell \in \mathbb{Z}_+$, the $[\mathrm{V} \ell \, | \, k \mathrm{A}]$-connected components can be found in polynomial time if each vertex is single-colored.

\begin{theorem} \label{thm:complexity_of_VlkA_single}
For any $k, \ell \in \mathbb{Z}_+$, the following problem can be solved in polynomial time.

 \smallskip
 
 \noindent {\bf \textsc{$\boldsymbol{[\mathrm{V} \ell \, | \, k \mathrm{A}]}$-ConnectedComponents--EachVertexSingleColored}}

 \smallskip
 
 \noindent \textit{Instance:} a digraph $G$, a finite color set $C$, a function $f \colon V(G) \to C$, and $m \in \mathbb{Z}_+$.
 
 \noindent \textit{Question:} is it true that $G$ has a $[\mathrm{V} \ell \, | \, k \mathrm{A}]$-connected component of size at least~$m$?

 \smallskip
 
 \noindent Moreover, in this single-colored case, all the $[\mathrm{V} \ell \, | \, k \mathrm{A}]$-connected components of $G$ can be found in polynomial time.
\end{theorem}

\subsubsection[\texorpdfstring{The $[\mathrm{V} \ell \, | \, k \mathrm{V}]$}{}-connected components when every vertex is single-colored]{The \texorpdfstring{$\boldsymbol{[\mathrm{V} \ell \, | \, k \mathrm{V}]}$}{}-connected components when every vertex is single-colored} \label{section:VlkV_old_components}

Similarly to $G_{[\mathrm{V} \ell \mid k \mathrm{A}]}$, the auxiliary graph $G_{[\mathrm{V} \ell \mid k \mathrm{V}]}$ has some special structure if every vertex is single-colored. Again, we defer the proof to the \hyperref[proof:property_of_VlkV_single]{Appendix}.

\begin{lemma} \label{lemma:property_of_VlkV_single}
 Let $k, \ell \in \mathbb{Z}_+$, let $G$ be a digraph, let $C$ be a finite color set, and let $f \colon V(G) \to C$. Then the auxiliary graph $G_{[\mathrm{V} \ell \mid k \mathrm{V}]}$ is $\overline{p K_2}$-free, where $p \colonequals \big\lceil (k \ell + k + \ell + 2)/2 \big\rceil$.
\end{lemma}

By Theorem~\ref{thm:pK2_free_listing_max_clique}, we conclude that for any fixed $k, \ell \in \mathbb{Z}_+$, the $[\mathrm{V} \ell \, | \, k \mathrm{V}]$-connected components can be found in polynomial time if each vertex is single-colored.

\begin{theorem}\label{thm:complexity_of_VlkV_single}
For any $k, \ell \in \mathbb{Z}_+$, the following problem can be solved in polynomial time.

 \smallskip
 
 \noindent {\bf \textsc{$\boldsymbol{[\mathrm{V} \ell \, | \, k \mathrm{V}]}$-ConnectedComponents--EachVertexSingleColored}}

 \smallskip
 
 \noindent \textit{Instance:} a digraph $G$, a finite color set $C$, a function $f \colon V(G) \to C$, and $m \in \mathbb{Z}_+$.
 
 \noindent \textit{Question:} is it true that $G$ has a $[\mathrm{V} \ell \, | \, k \mathrm{V}]$-connected component of size at least~$m$?

 \smallskip

 \noindent Moreover, in this single-colored case, all the $[\mathrm{V} \ell \, | \, k \mathrm{V}]$-connected components of $G$ can be found in polynomial time.
\end{theorem}

\subsection{Maximal vertex-color-avoiding strongly connected subgraphs} \label{section:Vlk_new_components}

In this section, we define maximal vertex-color-avoiding strongly connected subgraphs and prove that finding the maximal $[\mathrm{V} \ell \, | \, k \mathrm{A}]$- and the maximal $[\mathrm{V} \ell \, | \, k \mathrm{V}]$-connected subgraphs of a vertex-colored digraph is NP-hard. On the other hand, in Sections~\ref{section:VlkA_new_components} and~\ref{section:VlkV_new_components}, we show that the same problem in the setting where every vertex is single-colored can be solved in polynomial time.

Similarly to the case of $[\mathrm{A} \ell \, | \, k \mathrm{A}]$- or $[\mathrm{A} \ell \, | \, k \mathrm{V}]$-connectivity, it is not difficult to see that even though some set $W \subseteq V(G)$ of vertices is the vertex set of a $[\mathrm{V} \ell \, | \, k \mathrm{A}]$- or $[\mathrm{V} \ell \, | \, k \mathrm{V}]$-connected component of a vertex-colored digraph $G$, the subgraph $G[W]$ is not necessarily $[\mathrm{V} \ell \, | \, k \mathrm{A}]$- or $[\mathrm{V} \ell \, | \, k \mathrm{V}]$-connected.

In the following, we study such maximal vertex sets that span a vertex-color-avoiding connected subgraph. For a set $W \subseteq V(G)$ of vertices, the subgraph $G[W]$ is called a \emph{maximal vertex-$\ell$-color-avoiding strongly $k$-arc-connected subgraph}, or for short a \emph{maximal $[\mathrm{V} \ell \, | \, k \mathrm{A}]$-connected subgraph} (and \emph{maximal vertex-$\ell$-color-avoiding strongly $k$-vertex-connected subgraph}, or for short \emph{maximal $[\mathrm{V} \ell \, | \, k \mathrm{V}]$-connected subgraph}, respectively), of a vertex-colored digraph $G$, if it is inclusion-wise maximal for the property of being $[\mathrm{V} \ell \, | \, k \mathrm{A}]$-connected (or $[\mathrm{V} \ell \, | \, k \mathrm{V}]$-connected).

Similarly to finding vertex-color-avoiding strongly connected components, the problem of finding maximal vertex-color-avoiding strongly connected subgraphs is NP-hard. The proof is deferred to the \hyperref[proof:complexity_of_VlkA_and_VlkV_2]{Appendix}.

\begin{theorem} \label{thm:complexity_of_VlkA_and_VlkV_2}
 The following problems are NP-complete for any $k, \ell \in \mathbb{Z}_+$, even for undirected graphs.

 \smallskip
 
 \noindent {\bf \scshape \textsc{Maximal-$\boldsymbol{[\mathrm{V} \ell \, | \, k \mathrm{A}]}$-ConnectedSubgraphs}} \\
 \textit{Instance:} a digraph $G$, a finite color set $C$, a function $f \colon V(G) \to 2^C$, and $m \in \mathbb{Z}_+$. \\
 \textit{Question:} is it true that $G$ has a maximal $[\mathrm{V} \ell \, | \, k \mathrm{A}]$-connected subgraph on at least $m$ vertices?

 \medskip
 
 \noindent {\bf \scshape \textsc{Maximal-$\boldsymbol{[\mathrm{V} \ell \, | \, k \mathrm{V}]}$-ConnectedSubgraphs}} \\
 \textit{Instance:} a digraph $G$, a finite color set $C$, a function $f \colon V(G) \to 2^C$, and $m \in \mathbb{Z}_+$. \\
 \textit{Question:} is it true that $G$ has a maximal $[\mathrm{V} \ell \, | \, k \mathrm{V}]$-connected subgraph on at least $m$ vertices?
\end{theorem}
 
In Sections~\ref{section:VlkA_new_components} and~\ref{section:VlkV_new_components}, we consider the setting when every vertex is single-colored. In this special case, we present a polynomial-time algorithm for finding the maximal vertex-color-avoiding connected subgraphs of a given digraph for any fixed $k, \ell \in \mathbb{Z}_+$. To do so, we use the algorithms that can find the vertex-color-avoiding connected components when every vertex is single-colored as subroutines. Let \textbf{$\boldsymbol{[\mathrm{V} \ell \, | \, k \mathrm{A}]}$-Connected\-Components}$(G,f)$ and \textbf{$\boldsymbol{[\mathrm{V} \ell \, | \, k \mathrm{V}]}$-Connected\-Components}$(G,f)$ denote polynomial-time algorithms which respectively return the vertex sets of the $[\mathrm{V} \ell \, | \, k \mathrm{A}]$- and $[\mathrm{V} \ell \, | \, k \mathrm{V}]$-connected components of a digraph $G$ with vertex-coloring $f \colon V(G) \to C$, where $C$ is a finite color set. Note that by Theorems~\ref{thm:complexity_of_VlkA_single} and~\ref{thm:complexity_of_VlkV_single}, such polynomial-time algorithms exist. For ease of notation, let $f_W$ denote the restriction of $f$ to $W$ for any $W \subseteq V(G)$.

\subsubsection[\texorpdfstring{Maximal $[\mathrm{V} \ell \, | \, k \mathrm{A}]$}{}-connected subgraphs when every vertex is single-colored]{\texorpdfstring{Maximal $\boldsymbol{[\mathrm{V} \ell \, | \, k \mathrm{A}]}$}{}-connected subgraphs when every vertex is single-colored} \label{section:VlkA_new_components}

Now we consider the problem of finding the maximal $[\mathrm{V} \ell \, | \, k \mathrm{A}]$-connected subgraphs of a vertex-colored digraph in the setting when every vertex is single-colored. 

Similarly to $[\mathrm{A} \ell \, | \, k \mathrm{V}]$-connectivity, let us first study the pairwise intersections of the $[\mathrm{V} \ell \, | \, k \mathrm{A}]$-connected components. Since the proof of the following lemma is based on similar ideas to those of Lemma~\ref{lemma:AlkV_2_intersection}, we defer it to the \hyperref[proof:VlkA_2_intersection]{Appendix}.

\begin{lemma} \label{lemma:VlkA_2_intersection}
 Let $k, \ell \in \mathbb{Z}_+$, let $G$ be a digraph, let $C$ be a finite color set, and let $f \colon V(G) \to C$. Let $V_1, V_2 \subseteq V(G)$ be the vertex sets of two distinct $[\mathrm{V} \ell \, | \, k \mathrm{A}]$-connected components of $G$. Then, the following hold.
 \begin{enumerate}[topsep=3pt, itemsep=-1.5pt, label=(\roman*)]
  \item \label{item:VlkA_2_intersection_i} At most $\ell$ colors appear in $V_1 \cap V_2$.
  \item \label{item:VlkA_2_intersection_ii} There exists at least one color appearing in $V_1 \cup V_2$ but not in $V_1 \cap V_2$.
 \end{enumerate}
\end{lemma}

Now let us present our algorithm, which considers all sets of at most $\ell + 1$ vertices of pairwise distinct colors as follows. For every set of size exactly $\ell+1$, we take the $[\mathrm{V} \ell \, | \, k \mathrm{A}]$-connected components of $G$. If none of these components contains our set entirely, we proceed to the next set. Otherwise, by Lemma~\ref{lemma:VlkA_2_intersection}, there is exactly one such component; we then take the $[\mathrm{V} \ell \, | \, k \mathrm{A}]$-connected components of this component and continue this procedure until we obtain the maximal $[\mathrm{V} \ell \, | \, k \mathrm{A}]$-connected subgraphs containing at least $\ell + 1$ colors. We then consider all sets of at most $\ell$ vertices of pairwise distinct colors in decreasing order of size. For each such set, if it is not contained in any of the previously determined maximal $[\mathrm{V} \ell \, | \, k \mathrm{A}]$-connected subgraphs, we remove from $G$ all the vertices whose colors are different from those of the chosen vertices, and in the remaining graph, we again iteratively take the $[\mathrm{V} \ell \, | \, k \mathrm{A}]$-connected components of the $[\mathrm{V} \ell \, | \, k \mathrm{A}]$-connected component containing our chosen set as long as there exist such a component.

\bigskip

\begin{algorithm}[H] 
 \caption{Algorithm for finding the maximal $[\mathrm{V} \ell \, | \, k \mathrm{A}]$-connected subgraphs}
 \label{algo:VlkA_components_2}
 \DontPrintSemicolon
 
 \KwIn{a digraph $G$, a finite color set $C$, and a function $f \colon V(G) \to C$.}
 \KwOut{the vertex sets of the maximal $[\mathrm{V} \ell \, | \, k \mathrm{A}]$-connected subgraphs of $G$.}
 
 \smallskip

 $\mathcal{V}^* \gets \emptyset$\;
 \For{$W \subseteq V(G)$ with $|W| = \ell + 1$ and $|f(W)| = \ell + 1$}{
     $\mathcal{V} \gets [\mathrm{V} \ell \, | \, k \mathrm{A}] \textrm{-ConnectedComponents} \big( G, f \big)$\;
     \While{there exists $U \in \mathcal{V}$ for which $W \subseteq U$}{
         \If{$G[U]$ is $[\mathrm{V} \ell \, | \, k \mathrm{A}]$-connected}{
             $\mathcal{V}^* \gets \mathcal{V}^* \cup \{ U \}$\;
             \textbf{break}\;
         }
         \ElseIf{$U = W$}{
             \textbf{break}\;
         }
         \Else{
             $\mathcal{V} \gets [\mathrm{V} \ell \, | \, k \mathrm{A}] \textrm{-ConnectedComponents} \big( G[U], f_U \big)$
         }
     }
 }
 \For{$i = \ell$ \textbf{to} 1}{
     \For{$W \subseteq V(G)$ with $|W| = i$ and $|f(W)| = i$}{
         \If{there does not exist $U \in \mathcal{V}^*$ with $W \subset U$}{
             $W' \gets \big\{ v \in V(G) \bigm| f(v) \in f(W) \big\}$\;
             $\mathcal{V} \gets [\mathrm{V} \ell \, | \, k \mathrm{A}] \textrm{-ConnectedComponents} \big( G[W'], f_{W'} \big)$\;
             \While{there exists $U \in \mathcal{V}$ for which $W \subseteq U$}{
                 \If{$G[U]$ is $[\mathrm{V} \ell \, | \, k \mathrm{A}]$-connected}{
                     $\mathcal{V}^* \gets \mathcal{V}^* \cup \{ U \}$\;
                     \textbf{break}\;
                 }
                 \ElseIf{$U = W$}{
                     \textbf{break}\;
                 }
                 \Else{
                     $\mathcal{V} \gets [\mathrm{V} \ell \, | \, k \mathrm{A}] \textrm{-ConnectedComponents} \big( G[U], f_U \big)$
                 }
            }
         }
     }
 }
 \Return{$\mathcal{V}^*$}\;
\end{algorithm}

\bigskip

We show the correctness of Algorithm~\ref{algo:VlkA_components_2} and analyze its running time in the \hyperref[proof:VlkA_2_algo]{Appendix}.

\begin{theorem} \label{thm:VlkA_2_algo}
 Let $k, \ell \in \mathbb{Z}_+$ be fixed. Given a digraph $G$, a finite color set $C$ and a function $f \colon V(G) \to C$, Algorithm~\ref{algo:VlkA_components_2} outputs the vertex sets of the maximal $[\mathrm{V} \ell \, | \, k \mathrm{A}]$-connected subgraphs of $G$ and runs in polynomial time.
\end{theorem}

By Theorem~\ref{thm:VlkA_2_algo}, we conclude the following.

\begin{corollary} \label{cor:complexity_of_VlkA_single_2}
For any $k, \ell \in \mathbb{Z}_+$, the following problem can be solved in polynomial time.

 \smallskip
 
 \noindent {\bf \scshape \textsc{Maximal-$\boldsymbol{[\mathrm{V} \ell \, | \, k \mathrm{A}]}$-ConnectedSubgraphs-EachVertexSingleColored}} \\
 \textit{Instance:} a digraph $G$, a finite color set $C$, a function $f \colon V(G) \to C$, and $m \in \mathbb{Z}_+$. \\
 \textit{Question:} is it true that $G$ has a maximal $[\mathrm{V} \ell \, | \, k \mathrm{A}]$-connected subgraph on at least $m$ vertices?
 
 \smallskip
 
 \noindent Moreover, in this single-colored case, all the maximal $[\mathrm{V} \ell \, | \, k \mathrm{A}]$-connected subgraphs of $G$ can be found in polynomial time.
\end{corollary}

\subsubsection[\texorpdfstring{Maximal $[\mathrm{V} \ell \, | \, k \mathrm{V}]$}{}-connected subgraphs when every vertex is single-colored]{\texorpdfstring{Maximal $\boldsymbol{[\mathrm{V} \ell \, | \, k \mathrm{V}]}$}{}-connected subgraphs when every vertex is single-colored} \label{section:VlkV_new_components}

Now we consider the problem of finding the maximal $[\mathrm{V} \ell \, | \, k \mathrm{V}]$-connected subgraphs of a vertex-colored digraph in the setting when every vertex is single-colored. Similarly to $[\mathrm{V} \ell \, | \, k \mathrm{A}]$-connectivity, we first study the pairwise intersections of the $[\mathrm{V} \ell \, | \, k \mathrm{V}]$-connected components. Again, the proof is deferred to the \hyperref[proof:VlkV_2_intersection]{Appendix}.

\begin{lemma} \label{lemma:VlkV_2_intersection}
 Let $k, \ell \in \mathbb{Z}_+$, let $G$ be a digraph, let $C$ be a finite color set, and let $f \colon V(G) \to C$. Let $V_1, V_2 \subseteq V(G)$ be the vertex set of two distinct $[\mathrm{V} \ell \, | \, k \mathrm{V}]$-connected components of $G$. Then, the following hold.
 \begin{enumerate}[topsep=3pt, itemsep=-1.5pt, label=(\roman*)]
  \item \label{item:VlkV_2_intersection_i} At most $k + \ell - 1$ colors appear in $V_1 \cap V_2$.
  \item \label{item:VlkV_2_intersection_ii} There exists $U \subseteq V_1 \cap V_2$ with $|U| \le k-1$ such that at least one color appears in $(V_1 \cup V_2) \setminus U$ but not in $(V_1 \cap V_2) \setminus U$.
 \end{enumerate}
\end{lemma}

To find the maximal $[\mathrm{V} \ell \, | \, k \mathrm{V}]$-connected subgraphs, we propose an algorithm analogous to Algorithm~\ref{algo:VlkA_components_2}.

\bigskip

\begin{algorithm}[H] 
 \caption{Algorithm for finding the maximal $[\mathrm{V} \ell \, | \, k \mathrm{V}]$-connected subgraphs}
 \label{algo:VlkV_components_2}
 \DontPrintSemicolon
 
 \KwIn{a digraph $G$, a finite color set $C$, and a function $f \colon V(G) \to C$.}
 \KwOut{the vertex sets of the maximal $[\mathrm{V} \ell \, | \, k \mathrm{V}]$-connected subgraphs of $G$.}
 
 \smallskip

 $\mathcal{V}^* \gets \emptyset$\;
 \For{$W \subseteq V(G)$ with $|W| = k + \ell$ and $|f(W)| = k + \ell$}{
     $\mathcal{V} \gets [\mathrm{V} \ell \, | \, k \mathrm{V}] \textrm{-ConnectedComponents} \big( G, f \big)$\;
     \While{there exists $U \in \mathcal{V}$ for which $W \subseteq U$}{
         \If{$G[U]$ is $[\mathrm{V} \ell \, | \, k \mathrm{V}]$-connected}{
             $\mathcal{V}^* \gets \mathcal{V}^* \cup \{ U \}$\;
             \textbf{break}\;
         }
         \ElseIf{$U = W$}{
             \textbf{break}\;
         }
         \Else{
             $\mathcal{V} \gets [\mathrm{V} \ell \, | \, k \mathrm{V}] \textrm{-ConnectedComponents} \big( G[U], f_U \big)$
         }
     }
 }
 \For{$i = k + \ell - 1$ \textbf{to} 1}{
     \For{$W \subseteq V(G)$ containing exactly $i$ colors and at most $k$ vertices of each color}{
         $W' \gets W \cup \big\{ v \in V(G) \bigm| \text{$W$ contains exactly $k$ vertices of color $f(v)$} \big\}$\;
         $\mathcal{V} \gets [\mathrm{V} \ell \, | \, k \mathrm{V}] \textrm{-ConnectedComponents} \big( G[W'], f_{W'} \big)$\;
         \While{there exists $U \in \mathcal{V}$ for which $W \subseteq U$}{
             \If{$G[U]$ is $[\mathrm{V} \ell \, | \, k \mathrm{V}]$-connected}{
                 $\mathcal{V}^* \gets \mathcal{V}^* \cup \{ U \}$\;
                 \textbf{break}\;
             }
             \ElseIf{$U = W$}{
                 \textbf{break}\;
             }
             \Else{
                 $\mathcal{V} \gets [\mathrm{V} \ell \, | \, k \mathrm{V}] \textrm{-ConnectedComponents} \big( G[U], f_U \big)$
             }
         }
     }
 }
 $\mathcal{V}^* \gets \{ U \in \mathcal{V}^* \mid \text{$U$ is inclusion-wise maximal in $\mathcal{V}^*$} \}$\;
 \Return{$\mathcal{V}^*$}\;
\end{algorithm}

\bigskip

We show the correctness of Algorithm~\ref{algo:VlkV_components_2} and analyze its running time in the \hyperref[proof:VlkV_2_algo]{Appendix}.

\begin{theorem} \label{thm:VlkV_2_algo}
 Let $k, \ell \in \mathbb{Z}_+$ be fixed. Given a digraph $G$, a finite color set $C$, a function $f \colon V(G) \to C$, Algorithm~\ref{algo:VlkV_components_2} outputs the vertex sets of the maximal $[\mathrm{V} \ell \, | \, k \mathrm{V}]$-connected subgraphs of $G$ and runs in polynomial time.
\end{theorem}

By Theorem~\ref{thm:VlkV_2_algo}, we conclude the following.

\begin{corollary} \label{cor:complexity_of_VlkV_single_2}
For any $k, \ell \in \mathbb{Z}_+$, the following problem can be solved in polynomial time.

 \smallskip
 
 \noindent {\bf \scshape \textsc{Maximal-$\boldsymbol{[\mathrm{V} \ell \, | \, k \mathrm{V}]}$-ConnectedSubgraphs-EachVertexSingleColored}} \\
 \textit{Instance:} a digraph $G$, a finite color set $C$, a function $f \colon V(G) \to C$, and $m \in \mathbb{Z}_+$. \\
 \textit{Question:} is it true that $G$ has a maximal $[\mathrm{V} \ell \, | \, k \mathrm{V}]$-connected subgraph on at least $m$ vertices?
 
 \smallskip
 
 \noindent Moreover, in this single-colored case, all the maximal $[\mathrm{V} \ell \, | \, k \mathrm{V}]$-connected subgraphs of $G$ can be found in polynomial time.
\end{corollary}

\section{Internally vertex-color-avoiding strong connectivity} \label{section:IVlk}

In this section, we focus on internally vertex-color-avoiding strong connectivity. First, we present the formal definitions; then, in Section~\ref{section:IVlk_old_components}, we prove that finding the internally vertex-color-avoiding strongly connected components of a vertex-colored digraph is NP-hard even if every vertex is single-colored with the same color. Finally, in Section~\ref{section:IVlk_new_components}, we prove that finding the maximal internally vertex-color-avoiding strongly connected subgraphs is also NP-hard even if every vertex is single-colored with the same color.

\begin{definition}[${[\mathrm{IV} \ell \, | \, k \mathrm{A}]}$- and ${[\mathrm{IV} \ell \, | \, k \mathrm{V}]}$-connectivity] \label{def:IVlkA_and_IVlkV}
 Let $k, \ell \in \mathbb{Z}_+$, let $G$ be a digraph, let $C$ be a finite color set, and let $f \colon V(G) \to 2^C$ be a function assigning some colors to every vertex. Two distinct vertices $u, v \in V(G)$ are called \emph{internally vertex-$\ell$-color-avoiding strongly $k$-arc-connected}, or for short \emph{$[\mathrm{IV} \ell \, | \, k \mathrm{A}]$-connected} (and respectively, \emph{internally vertex-$\ell$-color-avoiding strongly $k$-vertex-connected}, or for short \emph{$[\mathrm{IV} \ell \, | \, k \mathrm{V}]$-connected}) if, after the removal of any at most $\ell$ colors excluding the vertices $u$ and $v$, they remain strongly $k$-arc-connected (or strongly $k$-vertex-connected).
\end{definition}

Observe that if two vertices are $[\mathrm{IV} \ell \, | \, k \mathrm{A}]$- or $[\mathrm{IV} \ell \, | \, k \mathrm{V}]$-connected, then they are also $[\mathrm{V} \ell \, | \, k \mathrm{A}]$- or $[\mathrm{V} \ell \, | \, k \mathrm{V}]$-connected, respectively, but the converse is not necessarily true. For example, in the vertex-colored digraph $G_1$ shown in Figure~\ref{fig:example_VlkA}, no two vertices are $[\mathrm{IV} 2 \, | \, 1 \mathrm{A}]$-connected.

\subsection{Internally vertex-color-avoiding strongly connected components} \label{section:IVlk_old_components}

In this section, we define internally vertex-color-avoiding strongly connected components and prove that finding the $[\mathrm{IV} \ell \, | \, k \mathrm{A}]$- and the $[\mathrm{IV} \ell \, | \, k \mathrm{V}]$-connected components of a vertex-colored digraph is NP-hard.

It is not difficult to see that the relations of $[\mathrm{IV} \ell \, | \, k \mathrm{A}]$- and $[\mathrm{IV} \ell \, | \, k \mathrm{V}]$-connectivity are not transitive, and thereby they are not equivalence relations for any $k, \ell \in \mathbb{Z}_+$.

For any $k, \ell \in \mathbb{Z}_+$, the \emph{internally vertex-$\ell$-color-avoiding strongly $k$-arc-} and \emph{$k$-vertex-connected components}, or for short the \emph{$[\mathrm{IV} \ell \, | \, k \mathrm{A}]$-} and \emph{$[\mathrm{IV} \ell \, | \, k \mathrm{V}]$-connected components}, of a vertex-colored digraph $G$ are maximal sets of vertices such that any two of these vertices are $[\mathrm{IV} \ell \, | \, k \mathrm{A}]$- or $[\mathrm{IV} \ell \, | \, k \mathrm{V}]$-connected in $G$, respectively. If a vertex $v \in V(G)$ is not $[\mathrm{IV} \ell \, | \, k \mathrm{A}]$- or $[\mathrm{IV} \ell \, | \, k \mathrm{V}]$-connected to any other vertex of the digraph $G$ for some $k, \ell \in \mathbb{Z}_+$, then $v$ is considered to form an $[\mathrm{IV} \ell \, | \, k \mathrm{A}]$- or $[\mathrm{IV} \ell \, | \, k \mathrm{V}]$-connected component of $G$ in itself, respectively.

We say that a vertex-colored digraph on at least two vertices is \emph{internally vertex-$\ell$-color-avoiding strongly $k$-arc-connected}, or for short $[\mathrm{IV} \ell \, | \, k \mathrm{A}]$-connected (and respectively, \emph{internally vertex-$\ell$-color-avoiding strongly $k$-vertex-connected}, or for short $[\mathrm{IV} \ell \, | \, k \mathrm{V}]$-connected) if it has exactly one $[\mathrm{IV} \ell \, | \, k \mathrm{A}]$-connected component (or $[\mathrm{IV} \ell \, | \, k \mathrm{V}]$-connected component). A vertex-colored digraph on exactly one vertex is $[\mathrm{IV} \ell \, | \, k \mathrm{A}]$-connected (and respectively, $[\mathrm{IV} \ell \, | \, k \mathrm{V}]$-connected) if it has at least $k$ loops.

To find the internally vertex-color-avoiding strongly connected components of a vertex-colored digraph, we use auxiliary graphs.

\begin{definition}[Auxiliary graphs $G_{[\mathrm{IV} \ell \mid k \mathrm{A}]}$ and $G_{[\mathrm{IV} \ell \mid k \mathrm{V}]}$]
 Let $k, \ell \in \mathbb{Z}_+$, let $G$ be a digraph, let $C$ be a finite color set, and let $f \colon V(G) \to 2^C$. The \emph{auxiliary graphs} of $G$ are the undirected graphs $G_{[\mathrm{IV} \ell \mid k \mathrm{A}]}$ and $G_{[\mathrm{IV} \ell \mid k \mathrm{V}]}$ on the same vertex set as $G$ and two vertices of theirs are adjacent if and only if the corresponding vertices in $G$ are $[\mathrm{IV} \ell \, | \, k \mathrm{A}]$- or $[\mathrm{IV} \ell \, | \, k \mathrm{V}]$-connected, respectively.
\end{definition}

Similarly as before, we have that $W \subseteq V(G)$ is the vertex set of an $[\mathrm{IV} \ell \, | \, k \mathrm{A}]$- or an $[\mathrm{IV} \ell \, | \, k \mathrm{V}]$-connected component of $G$ if and only if $W$ spans a maximal clique in $G_{[\mathrm{IV} \ell \mid k \mathrm{A}]}$ or $G_{[\mathrm{IV} \ell \mid k \mathrm{V}]}$, respectively.

The proof that the problem of finding internally vertex-color-avoiding strongly connected components of a vertex-colored digraph is NP-hard is similar to that of Theorem~\ref{thm:complexity_of_VlkA_and_VlkV}, and is thereby deferred to the \hyperref[proof:complexity_of_IVlkA_and_IVlkV]{Appendix}.

\begin{theorem} \label{thm:complexity_of_IVlkA_and_IVlkV}
For any $k, \ell \in \mathbb{Z}_+$, the following problems are NP-complete even for undirected graphs in which every vertex is single-colored with the same color.

 \smallskip

 \noindent {\bf \textsc{$\boldsymbol{[\mathrm{IV} \ell \, | \, k \mathrm{A}]}$-ConnectedComponents}}

 \smallskip
 
 \noindent \textit{Instance:} a digraph $G$, a finite color set $C$, a function $f \colon V(G) \to 2^C$, and $m \in \mathbb{Z}_+$.
 
 \noindent textit{Question:} is it true that $G$ has a $[\mathrm{IV} \ell \, | \, k \mathrm{A}]$-connected component of size at least~$m$?

 \medskip
 
 \noindent {\bf \textsc{$\boldsymbol{[\mathrm{IV} \ell \, | \, k \mathrm{V}]}$-ConnectedComponents}}

 \smallskip
 
 \noindent \textit{Instance:} a digraph $G$, a finite color set $C$, a function $f \colon V(G) \to 2^C$, and $m \in \mathbb{Z}_+$.
 
 \noindent \textit{Question:} is it true that $G$ has a $[\mathrm{IV} \ell \, | \, k \mathrm{V}]$-connected component of size at least~$m$?
\end{theorem}

\subsection{Maximal internally vertex-color-avoiding strongly connected subgraphs}\label{section:IVlk_new_components}

In this section, we define maximal internally vertex-color-avoiding strongly connected subgraphs and prove that finding the maximal $[\mathrm{IV} \ell \, | \, k \mathrm{A}]$- and the maximal $[\mathrm{IV} \ell \, | \, k \mathrm{V}]$-connected subgraphs of a vertex-colored digraph is NP-hard.

Similarly as before, it is not difficult to see that even though some set $W \subseteq V(G)$ of vertices is the vertex set of an $[\mathrm{IV} \ell \, | \, k \mathrm{A}]$- or $[\mathrm{IV} \ell \, | \, k \mathrm{V}]$-connected component of a vertex-colored digraph $G$, the subgraph $G[W]$ is not necessarily $[\mathrm{IV} \ell \, | \, k \mathrm{A}]$- or $[\mathrm{IV} \ell \, | \, k \mathrm{V}]$-connected.

In the following, we study such maximal vertex sets that span an internally vertex-color-avoiding connected subgraph. For a set $W \subseteq V(G)$ of vertices, the subgraph $G[W]$ is called a \emph{maximal internally vertex-$\ell$-color-avoiding strongly $k$-arc-connected subgraph}, or for short a \emph{maximal $[\mathrm{IV} \ell \, | \, k \mathrm{A}]$-connected subgraph} (and \emph{maximal internally vertex-$\ell$-color-avoiding strongly $k$-vertex-connected subgraph}, or for short \emph{maximal $[\mathrm{IV} \ell \, | \, k \mathrm{V}]$-connected subgraph}, respectively), of a vertex-colored digraph $G$, if it is inclusion-wise maximal for the property of being $[\mathrm{IV} \ell \, | \, k \mathrm{A}]$-connected (or $[\mathrm{IV} \ell \, | \, k \mathrm{V}]$-connected).

In the \hyperref[proof:complexity_of_IVlkA_and_IVlkV_2]{Appendix}, we prove that the problem of finding maximal internally vertex-color-avoiding strongly connected subgraphs is NP-hard.

\begin{theorem} \label{thm:complexity_of_IVlkA_and_IVlkV_2}
 The following problems are NP-complete for any $k, \ell \in \mathbb{Z}_+$c, even for undirected graphs in which every vertex has the same color.

 \smallskip
 
 \noindent {\bf \scshape \textsc{Maximal-$\boldsymbol{[\mathrm{IV} \ell \, | \, k \mathrm{A}]}$-ConnectedSubgraphs}} \\
 \textit{Instance:} a digraph $G$, a finite color set $C$, a function $f \colon V(G) \to 2^C$, and $m \in \mathbb{Z}_+$. \\
 \textit{Question:} is it true that $G$ has a maximal $[\mathrm{IV} \ell \, | \, k \mathrm{A}]$-connected subgraph on at least $m$ vertices?

 \medskip
 
 \noindent {\bf \scshape \textsc{Maximal-$\boldsymbol{[\mathrm{IV} \ell \, | \, k \mathrm{V}]}$-ConnectedSubgraphs}} \\
 \textit{Instance:} a digraph $G$, a finite color set $C$, a function $f \colon V(G) \to 2^C$, and $m \in \mathbb{Z}_+$. \\
 \textit{Question:} is it true that $G$ has a maximal $[\mathrm{IV} \ell \, | \, k \mathrm{V}]$-connected subgraph on at least $m$ vertices?
\end{theorem}

\section{Conclusion}

In this paper, we showed that arc-$\ell$-color-avoiding strongly $k$-arc- and $k$-vertex-connected components, as well as maximal arc-$\ell$-color-avoiding strongly $k$-arc- and $k$-vertex-connected subgraphs, can be found in polynomial time. In contrast, for vertex- and internally vertex-$\ell$-color-avoiding $k$-arc- and $k$-vertex-connectivity, these problems become NP-hard. However, when each vertex is assigned a single color, the problem of finding vertex-$\ell$-color-avoiding strongly $k$-arc- and $k$-vertex-connected components, as well as maximal vertex-$\ell$-color-avoiding strongly $k$-arc- and $k$-vertex-connected subgraphs can be solved in polynomial time.

Determining the complexity of finding vertex-$\ell$-color-avoiding strongly $k$-arc- and $k$-vertex-connected components and maximal such subgraphs when each vertex has a constant-sized color set remains open.

We also note that throughout this paper we always worked with fixed values of $k$ and $\ell$. However, determining the computational complexity of the problems shown to be polynomial-time solvable here when $k$ and $\ell$ are part of the input remains also open.

\paragraph{Acknowledgement.} The authors are grateful to Krist\'{o}f B\'{e}rczi for his comments regarding the manuscript. The authors would also like to express their gratitude to M\'{a}rk Penc for useful conversations. The research presented in this paper was supported by the Ministry of Culture and Innovation and the National Research, Development and Innovation Office within the framework of the Artificial Intelligence National Laboratory Programme. The research of Kitti Varga was supported by the Ministry of Innovation and Technology of Hungary from the National Research, Development and Innovation Fund -- grant STARTING 153093.

\appendix

\section*{Appendix} \label{appendix}

Here we present the deferred proofs.

\medskip

First, we prove that the auxiliary graph $G_{[\mathrm{V} \ell \mid k \mathrm{A}]}$ is $\overline{(\ell + 2) K_2}$-free.

\begin{proof}[Proof of Lemma~\ref{lemma:property_of_VlkA_single}] \label{proof:property_of_VlkA_single}
 Suppose to the contrary that $G_{[\mathrm{V} \ell \mid k \mathrm{A}]}$ contains an induced $\overline{(\ell + 2) K_2}$ subgraph. Let $v_1, \ldots, v_{2 \ell + 4} \allowbreak \in V \big( G_{[\mathrm{V} \ell \mid k \mathrm{A}]} \big)$ denote the vertices of this induced $\overline{(\ell + 2) K_2}$ such that $v_{2i-1} v_{2i} \notin E \big( \, \overline{(\ell + 2) K_2} \, \big)$ for any $i \in [\ell + 2]$.

 By a previous observation regarding Definition~\ref{def:VlkA_and_VlkV}, see~\ref{item:VlkA}, since $v_{2 \ell + 3} v_{2 \ell + 4} \notin E \big( G_{[\mathrm{V} \ell \mid k \mathrm{A}]} \big)$, there exist $C' \subseteq C \setminus \big\{ f(v_{2 \ell + 3}), f(v_{2 \ell + 4}) \big\}$ and $A' \subseteq A(G)$ with $|C'| \le \ell$ and $|A'| \le k-1$ such that after the removal of $C'$ and $A'$ from $G$ there does not remain any $v_{2 \ell + 3}$-$v_{2 \ell + 4}$ or any $v_{2 \ell + 4}$-$v_{2 \ell + 3}$ dipath.
 
 Now we show that $v_1, \ldots, v_{2 \ell + 2}$ are all of some colors from $C'$. Otherwise, without loss of generality, we could assume that $v_1$ is not of a color from $C'$. Then after the removal of $C'$ and $A'$ from $G$, there would remain a $v_{2 \ell + 3}$-$v_1$, a $v_1$-$v_{2 \ell + 3}$, a $v_1$-$v_{2 \ell + 4}$, and a $v_{2 \ell + 4}$-$v_1$ dipath. Hence there would also remain a directed walk between $v_{2 \ell + 3}$ and $v_{2 \ell + 4}$ in both directions via $v_1$, contradicting $v_{2 \ell + 3} v_{2 \ell + 4} \notin E \big( G_{[\mathrm{V} \ell \mid k \mathrm{A}]} \big)$.

 Then the pigeonhole principle and $|C'| \le \ell$ imply that at least three of the vertices $v_1, \ldots, v_{2 \ell + 2}$ are of the same color. Without loss of generality, we can assume that $v_1$ and $v_3$ are of the same color. By the previous observation regarding Definition~\ref{def:VlkA_and_VlkV}, see~\ref{item:VlkA}, since $v_1 v_2 \notin E \big( G_{[\mathrm{V} \ell \mid k \mathrm{A}]} \big)$, there exist $C'' \subseteq C \setminus \big\{ f(v_1), f(v_2) \big\}$ and $A'' \subseteq A(G)$ with $|C''| \le \ell$ and $|A''| \le k-1$ such that after the removal of $C''$ and $A''$ from $G$ there does not remain any $v_1$-$v_2$ or any $v_2$-$v_1$ dipath. However, since $v_1 v_3, v_3 v_2 \in E \big( G_{[\mathrm{V} \ell \mid k \mathrm{A}]} \big)$ and $f(v_3) = f(v_1) \notin C''$, after the removal $C''$ and $A''$ from $G$ there remains a $v_1$-$v_3$, a $v_3$-$v_1$, a $v_3$-$v_2$, and a $v_2$-$v_3$ dipath. Hence there also remains a directed walk between $v_1$ and $v_2$ in both directions via $v_3$, which is a contradiction.
\end{proof}

Now we prove that the auxiliary graph $G_{[\mathrm{V} \ell \mid k \mathrm{V}]}$ is $\overline{p K_2}$-free with $p = \big\lceil (k \ell + k + \ell + 2)/2 \big\rceil$.

\begin{proof}[Proof of Lemma~\ref{lemma:property_of_VlkV_single}] \label{proof:property_of_VlkV_single}
 Suppose to the contrary that $G_{[\mathrm{V} \ell \mid k \mathrm{V}]}$ contains an induced $\overline{p K_2}$ subgraph. Let $v_1, \ldots, \allowbreak v_{2p} \in V \big( G_{[\mathrm{V} \ell \mid k \mathrm{V}]} \big)$ denote the vertices of this induced $\overline{p K_2}$ such that $v_{2i-1} v_{2i} \notin E \big( \, \overline{p K_2} \, \big)$ for any $i \in [p]$.

By a previous observation regarding Definition~\ref{def:VlkA_and_VlkV}, see~\ref{item:VlkV}, since $v_{2p-1} v_{2p} \notin E \big( G_{[\mathrm{V} \ell \mid k \mathrm{V}]} \big)$, there exist $C' \subseteq C \setminus \big\{ f(v_{2p-1}), f(v_{2p}) \big\}$ and $A' \subseteq A(G)$ and $V' \subseteq V(G) \setminus \{ v_{2p-1}, v_{2p} \}$ with $|C'| \le \ell$ and $|A'| + |V'| \le k-1$ such that after the removal of $C'$, $A'$, and $V'$ from $G$ there does not remain any $v_{2p-1}$-$v_{2p}$ or any $v_{2p}$-$v_{2p-1}$ dipath. 
 
 Since $2p - 2 \ge k \ell + k + \ell$ and $|V'| \le k-1$, we have $\big| \{ v_1, \ldots, v_{2 p-2} \} \setminus V' \big| \geq k \ell + \ell + 1$.
 
 First, we show that the vertices of $\{ v_1, \ldots, v_{2p-2} \} \setminus V'$ are all of some colors from $C'$. Otherwise, without loss of generality, we could assume that $v_1$ is not of a color from $C'$. Then after the removal of $C'$, $A'$, and $V'$ from $G$, there would remain a $v_{2p-1}$-$v_1$, a $v_1$-$v_{2p-1}$, a $v_1$-$v_{2p}$, and a $v_{2p}$-$v_1$ dipath. Hence there would also remain a directed walk between $v_{2p-1}$ and $v_{2p}$ in both directions via $v_1$, contradicting $v_{2p-1} v_{2p} \notin E \big( G_{[\mathrm{V} \ell \mid k \mathrm{V}]} \big)$.

 Then the pigeonhole principle and $|C'| \le \ell$ imply that there exists $W \subseteq \{ v_1, \ldots, v_{2p-2} \} \setminus V'$ such that $|W| \ge k+2$ and all vertices of $W$ are of the same color. Without loss of generality, we can assume that $v_1 \in W$. 
 
By the previous observation regarding Definition~\ref{def:VlkA_and_VlkV}, see~\ref{item:VlkV}, since $v_1 v_2 \notin E \big( G_{[\mathrm{V} \ell \mid k \mathrm{V}]} \big)$, there exist $C'' \subseteq C \setminus \big\{ f(v_1), f(v_2) \big\}$, and $A'' \subseteq A(G)$, and $V'' \subseteq V \setminus \{ v_1, v_2 \}$ with $|C''| \le \ell$ and $|A''| + |V''| \le k-1$ such that after the removal of $C''$, $A''$, and $V''$ from $G$ there does not remain any $v_1$-$v_2$ or any $v_2$-$v_1$ dipath. 

 Note that $W \setminus \big( V'' \cup \{ v_1, v_2 \} \big) \ne \emptyset$; without loss of generality, we can assume that $v_3 \in W \setminus \big( V'' \cup \{ v_1, v_2 \} \big)$.
 However, since $v_1 v_3, v_3 v_2 \in E \big( G_{[\mathrm{V} \ell \mid k \mathrm{V}]} \big)$ and $f(v_3) = f(v_1) \notin C''$ and $v_3 \notin V''$, after the removal $C''$, $A''$, and $V''$ from $G$ there remains a $v_1$-$v_3$, a $v_3$-$v_1$, a $v_3$-$v_2$, and a $v_2$-$v_3$ dipath. Hence there also remains a directed walk between $v_1$ and $v_2$ in both directions via $v_3$, which is a contradiction.
\end{proof}

Now we prove that given a vertex-colored digraph and a positive integer $m$, the problems of deciding if the digraph has a maximal $[\mathrm{V} \ell \, | \, k \mathrm{A}]$- or $[\mathrm{V} \ell \, | \, k \mathrm{V}]$-connected subgraph of size at least $m$ are both NP-complete.

\begin{proof}[Proof of Theorem~\ref{thm:complexity_of_VlkA_and_VlkV_2}] \label{proof:complexity_of_VlkA_and_VlkV_2}
 Let $k, \ell \in \mathbb{Z}_+$ be fixed. Obviously, these problems are in NP.
 
 Now we show that these problems are NP-hard by reducing the NP-complete problem \textsc{Clique}~\cite{karp1972reducibility} to them. Let $G$ be an arbitrary undirected graph. Let $G'$ be the undirected graph constructed in the proof of Theorem~\ref{thm:complexity_of_VlkA_and_VlkV}, i.e., $G'$ is obtained from $G$ by replacing each of its edges with $k$ parallel edges. Assign a color to any two vertices of $G'$, and add this color to the set of colors of every other vertex; so altogether we use $\binom{n}{2}$ colors and every vertex has $\binom{n}{2} - (n-1)$ colors in its set of colors.
 
 It is not difficult to see that the following are equivalent for any $W \subseteq V(G) = V(G')$ with $|W| \ge 2$:
 \begin{itemize}[topsep=3pt, itemsep=-1.5pt]
  \item[--] $G[W]$ is a clique,
  \item[--] the subgraph $G'[W]$ is $[\mathrm{V} \ell \, | \, k \mathrm{A}]$-connected,
  \item[--] the subgraph $G'[W]$ is $[\mathrm{V} \ell \, | \, k \mathrm{V}]$-connected.
 \end{itemize}
  
 Therefore, the problems \textsc{Maximal-$[\mathrm{V} \ell \, | \, k \mathrm{A}]$-} and \textsc{Maximal-$[\mathrm{V} \ell \, | \, k \mathrm{V}]$-ConnectedSubgraphs} are indeed NP-complete. 
\end{proof}

Now we show that for any two distinct $[\mathrm{V} \ell \, | \, k \mathrm{A}]$-connected components, (i) at most $\ell$ colors appear in their intersection, and (ii) there exists at least one color appearing in their union but not in their intersection.

\begin{proof}[Proof of Lemma~\ref{lemma:VlkA_2_intersection}] \label{proof:VlkA_2_intersection}
 Let $V_1$ and $V_2$ be the vertex sets of two $[\mathrm{V} \ell \, | \, k \mathrm{A}]$-connected components.

 \smallskip

 Let us start with proving \ref{item:VlkA_2_intersection_i}. We show that if at least $\ell+1$ colors appear in $V_1 \cap V_2$, then any two vertices of $V_1 \cup V_2$ are $[\mathrm{V} \ell \, | \, k \mathrm{A}]$-connected. Then, by the inclusion-wise maximality of $V_1$ and $V_2$, we have $V_1 = V_2$. In order to see this, assume that there exist $\ell + 1$ distinct vertices $v_1, \ldots, v_{\ell + 1} \in V_1 \cap V_2$ with all distinct colors, and suppose to the contrary that there exist two vertices $u, w \in V_1 \cup V_2$, and subsets $C' \subseteq C \setminus \{ f(u), f(w) \big\}$ and $A' \subseteq A(G)$ with $|C'| \le \ell$ and $|A'| \le k-1$ such that after the removal of $C'$ and $A'$ from $G$ there does not remain any $u$-$w$ or any $w$-$u$ dipath. Without loss of generality, we can assume that $u \in V_1$. By the pigeonhole principle, there exists $i \in [\ell + 1]$ such that $f(v_i) \notin C'$. Note that it might happen that $u = v_i$. Since $V_1$ is the vertex set of a $[\mathrm{V} \ell \, | \, k \mathrm{A}]$-connected component, after the removal of $C'$ and $A'$ from $G$, there remains a $u$-$v_i$ and a $v_i$-$u$ dipath. (If $u = v_i$, then these dipaths have length 0.) Analogously, the same holds for $v_i$ and $w$. Therefore, after the removal of $C'$ and $A'$ from $G$, there remains a directed $u$-$w$ and a directed $w$-$u$ walk via $v_i$, which is a contradiction.

 \smallskip
 
 Now we prove \ref{item:VlkA_2_intersection_ii}. We show that if all the colors appearing in $V_1 \cup V_2$ also appear in $V_1 \cap V_2$, then any two vertices of $V_1 \cup V_2$ are $[\mathrm{V} \ell \, | \, k \mathrm{A}]$-connected. Then, by the inclusion-wise maximality of $V_1$ and $V_2$, we have $V_1 = V_2$. Suppose to the contrary that there exist two vertices $u, w \in V_1 \cup V_2$, and subsets $C' \subseteq C \setminus \big\{ f(u), f(w) \big\}$ and $A' \subseteq A(G)$ with $|C'| \le \ell$ and $|A'| \le k-1$ such that after the removal of $C'$ and $A'$ from $G$ there does not remain any $u$-$w$ or any $w$-$u$ dipath. Without loss of generality, we can assume that $u \in V_1$. Since all the colors appearing in $V_1 \cup V_2$ also appear in $V_1 \cap V_2$, there exists $v \in V_1 \cap V_2$ such that $f(v) = f(u) \notin C'$. Again, it might happen that $u = v$. Since $V_1$ is the vertex set of a $[\mathrm{V} \ell \, | \, k \mathrm{A}]$-connected component, after the removal of $C'$ and $A'$ from $G$, there remains a $u$-$v$ and a $v$-$u$ dipath. Analogously, after the removal of $C'$ and $A'$ from $G$, there remain a $v$-$w$ and a $w$-$v$ dipath. Therefore, after the removal of $C'$ and $A'$ from $G$, there remains a directed $u$-$w$ and a directed $w$-$u$ walk via $v$, which is a contradiction.
\end{proof}

Here, we discuss the correctness of Algorithm~\ref{algo:VlkA_components_2} and analyze its running time.

\begin{proof}[Proof of Theorem~\ref{thm:VlkA_2_algo}] \label{proof:VlkA_2_algo}
 We discuss the correctness and the time complexity of the algorithm separately.

 \smallskip
 
 \noindent \textit{Correctness.} The correctness of the algorithm directly follows from Lemma~\ref{lemma:VlkA_2_intersection}.

 \smallskip
 
 \noindent \textit{Time complexity.} It is not difficult to see that the for loops altogether take $\sum_{i=1}^{\ell+1} \binom{n}{i}$ iterations, and each iteration invokes the subroutine $[\mathrm{V} \ell \, | \, k \mathrm{A}]$-Connected\-Components on some subgraph of $G$ at most $n$ times.
 
 Therefore, the subroutine $[\mathrm{V} \ell \, | \, k \mathrm{A}]$-Connected\-Components on some subgraph of $G$ is invoked $O(n^{\ell+2})$ times, and by Theorem~\ref{thm:complexity_of_VlkA_single}, this subroutine has polynomial running time.
 
 Hence, Algorithm~\ref{algo:VlkA_components_2} runs in polynomial time.
\end{proof}

Now we show that for any two distinct $[\mathrm{V} \ell \, | \, k \mathrm{V}]$-connected components, (i) at most $k + \ell - 1$ colors appear in their intersection, and (ii) there exist at most $k-1$ vertices in their intersection such that at least one color appears among the remaining vertices in the union of the two components but not among the remaining vertices in the intersection.

\begin{proof}[Proof of Lemma~\ref{lemma:VlkV_2_intersection}] \label{proof:VlkV_2_intersection}
 Let $V_1$ and $V_2$ be the vertex sets of two $[\mathrm{V} \ell \, | \, k \mathrm{V}]$-connected components. 
 
 Let us start with proving \ref{item:VlkV_2_intersection_i}. We show that if at least $k + \ell$ colors appear in $V_1 \cap V_2$, then any two vertices of $V_1 \cup V_2$ are $[\mathrm{V} \ell \, | \, k \mathrm{V}]$-connected. Then, by the inclusion-wise maximality of $V_1$ and $V_2$, we have $V_1 = V_2$. In order to see this, assume that there exist $k + \ell$ distinct vertices $v_1, \ldots, v_{k + \ell} \in V_1 \cap V_2$ with all distinct colors, and suppose to the contrary that there exist two vertices $u, w \in V_1 \cup V_2$, and subsets $C' \subseteq C \setminus \{ f(u), f(w) \big\}$ and $A' \subseteq A(G)$ and $V' \subseteq V(G) \setminus \{ u, w \}$ with $|C'| \le \ell$ and $|A'| + |V'| \le k-1$ such that after the removal of $C'$, $A'$, and $V'$ from $G$ there does not remain any $u$-$w$ or any $w$-$u$ dipath. Without loss of generality, we can assume that $u \in V_1$. By the pigeonhole principle, there exists $i \in [k + \ell]$ such that $f(v_i) \notin C'$ and $v_i \notin V'$. Note that it might happen that $u = v_i$. Since $V_1$ is the vertex set of a $[\mathrm{V} \ell \, | \, k \mathrm{V}]$-connected component, after the removal of $C'$, $A'$, and $V'$ from $G$, there remains a $u$-$v_i$ and a $v_i$-$u$ dipath. (If $u = v_i$, then these dipaths have length~0.) Analogously, the same holds for $v_i$ and $w$. Therefore, after the removal of $C'$, $A'$, and $V'$ from $G$, there remains a directed $u$-$w$ and a directed $w$-$u$ walk via $v_i$, which is a contradiction.

 \smallskip
 
 Now we prove \ref{item:VlkV_2_intersection_ii}. We show that if for any $U \subseteq V_1 \cap V_2$ with $|U| \le k-1$ all the colors appearing in $(V_1 \cup V_2) \setminus U$ also appear in $(V_1 \cap V_2) \setminus U$, then any two vertices of $V_1 \cup V_2$ are $[\mathrm{V} \ell \, | \, k \mathrm{V}]$-connected. Then, by the inclusion-wise maximality of $V_1$ and $V_2$, we have $V_1 = V_2$.

 Suppose to the contrary that there exist two vertices $u, w \in V_1 \cup V_2$, and subsets $C' \subseteq C \setminus \{ f(u), f(w) \big\}$ and $A' \subseteq A(G)$ and $V' \subseteq V(G) \setminus \{ u, w \}$ with $|C'| \le \ell$ and $|A'| + |V'| \le k-1$ such that after the removal of $C'$, $A'$, and $V'$ from $G$ there does not remain any $u$-$w$ or any $w$-$u$ dipath. Without loss of generality, we can assume that $u \in V_1$. Since all the colors appearing in $(V_1 \cup V_2) \setminus V'$ also appear in $(V_1 \cap V_2) \setminus V'$, there exists $v \in (V_1 \cap V_2) \setminus V'$ such that $f(v) = f(u) \notin C'$. Again, it might happen that $u = v$. Since $V_1$ is the vertex set of a $[\mathrm{V} \ell \, | \, k \mathrm{V}]$-connected component, after the removal of $C'$, $A'$ and $V'$ from $G$, there remains a $u$-$v$ and a $v$-$u$ dipath. Analogously, after the removal of $C'$, $A'$, and $V'$ from $G$, there remains a $v$-$w$ and a $w$-$v$ dipath. Therefore, after the removal of $C'$, $A'$, and $V'$ from $G$, there remains a directed $u$-$w$ and a directed $w$-$u$ walk via $v$, which is a contradiction.
\end{proof}

Let us continue with the correctness and the running time of Algorithm~\ref{algo:VlkV_components_2}.

\begin{proof}[Proof of Theorem~\ref{thm:VlkV_2_algo}] \label{proof:VlkV_2_algo}
 We discuss the correctness and the time complexity of the algorithm separately.

 \smallskip
 
 \noindent \textit{Correctness.} The correctness of the algorithm directly follows from Lemma~\ref{lemma:VlkV_2_intersection} and the construction of $W'$.

 \smallskip
 
 \noindent \textit{Time complexity.} It is not difficult to see that the for loops altogether take at most $\binom{n}{k + \ell} + \sum_{j=1}^{k (k + \ell - 1)} \binom{n}{j}$ iterations, and each iteration invokes the subroutine $[\mathrm{V} \ell \, | \, k \mathrm{V}]$-Connected\-Components on some subgraph of $G$ at most $n$ times.
 
 Therefore, the subroutine $[\mathrm{V} \ell \, | \, k \mathrm{V}]$-Connected\-Components on some subgraph of $G$ is invoked

 \vspace{-7pt}
 
 \[ O(n^{\max \{ k + \ell, k (k + \ell - 1) \} + 1}) \]

 \vspace{-2pt}
 
 \noindent times, and by Theorem~\ref{thm:complexity_of_VlkV_single}, this subroutine has polynomial running time.

 In the final check of inclusion-wise maximality, at most $\binom{n}{k + \ell} + \sum_{j=1}^{k (k + \ell - 1)} \binom{n}{j}$ sets are considered. So this check also takes polynomial time.
 
 Hence, Algorithm~\ref{algo:VlkV_components_2} runs in polynomial time.
\end{proof}

Now we prove that given a vertex-colored digraph and a positive integer $m$, the problems of deciding if the digraph has an $[\mathrm{IV} \ell \, | \, k \mathrm{A}]$- or $[\mathrm{IV} \ell \, | \, k \mathrm{V}]$-connected component of size at least $m$ are both NP-complete  -- even for undirected graphs in which every vertex has the same color.

\begin{proof}[Proof of Theorem~\ref{thm:complexity_of_IVlkA_and_IVlkV}] \label{proof:complexity_of_IVlkA_and_IVlkV}
 Let $k, \ell \in \mathbb{Z}_+$ be fixed. Obviously, these problems are in NP.
 
 Now we show that these problems are NP-hard by reducing the NP-complete problem \textsc{Clique}~\cite{karp1972reducibility} to them. Let $G$ be an arbitrary undirected graph. Let $G'$ be the undirected graph that can be obtained from $G$ by replacing each of its edges with $k$ parallel edges. Let $C = \{ c \}$ and let $f$ be the function that assigns the color $c$ to each vertex of $G'$. For an example, see Figure~\ref{fig:example_IVlkA_and_IVlkV_complexity_construction}.
 
 \begin{figure}[H]
 \centering
 \begin{tikzpicture}[scale=0.8]
  \tikzstyle{vertex}=[draw,circle,fill,minimum size=6,inner sep=0]
 
  \tikzstyle{vertex_red}=[draw, thick, red!50!black, shape=rectangle, fill=red!50, minimum size=8pt, inner sep=0]
 
  \node at (-1.75,1.5) {$G$};
  
  \node[vertex] (v1) at (45:1) [label=above right:$v_1$] {};
  \node[vertex] (v2) at (135:1) [label=above left:$v_2$] {};
  \node[vertex] (v3) at (225:1) [label=below left:$v_3$] {};
  \node[vertex] (v4) at (315:1) [label=below right:$v_4$] {};
  
  \draw[very thick] (v1) -- (v2) -- (v3) -- (v4) -- (v1);
  
  \begin{scope}[shift={(5,0)}]
  \node at (-1.75,1.5) {$G'$};
 
  \node[vertex_red] (w1) at (45:1) [label=above right:$v_1$] {};
  \node[vertex_red] (w2) at (135:1) [label=above left:$v_2$] {};
  \node[vertex_red] (w3) at (225:1) [label=below left:$v_3$] {};
  \node[vertex_red] (w4) at (315:1) [label=below right:$v_4$] {};
  
  \draw[very thick] (w1) to [bend right=20] (w2);
  \draw[very thick] (w2) to [bend right=20] (w1);

  \draw[very thick] (w2) to [bend right=20] (w3);
  \draw[very thick] (w3) to [bend right=20] (w2);

  \draw[very thick] (w3) to [bend right=20] (w4);
  \draw[very thick] (w4) to [bend right=20] (w3);

  \draw[very thick] (w4) to [bend right=20] (w1);
  \draw[very thick] (w1) to [bend right=20] (w4);
  \end{scope}
 \end{tikzpicture}
 \caption{The vertex-colored undirected graph $G'$ constructed in the proof of Theorem~\ref{thm:complexity_of_IVlkA_and_IVlkV} to $G$ in the case $k = 2$ (and with $\ell \in \mathbb{Z}_+$ being arbitrary). We assign the color red (denoted by squares) to all the vertices of $G'$.}
 \label{fig:example_IVlkA_and_IVlkV_complexity_construction}
 \end{figure}
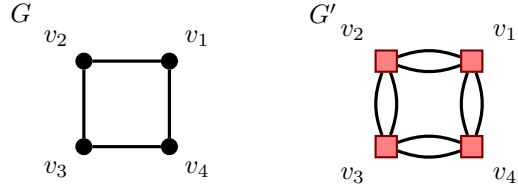
 
 It is not difficult to see that $G'_{[\mathrm{IV} \ell \, | \, k \mathrm{A}]} = G'_{[\mathrm{IV} \ell \, | \, k \mathrm{V}]} = G$, thus the following are equivalent:
 \begin{itemize}[topsep=3pt, itemsep=-1.5pt]
  \item[--] $G$ has a clique of size at least $m$, \item[--] $G'$ has an $[\mathrm{IV} \ell \, | \, k \mathrm{A}]$-connected component of size at least $m$, and
  \item[--] $G'$ has an $[\mathrm{IV} \ell \, | \, k \mathrm{V}]$-connected component of size at least $m$.
 \end{itemize}
  Therefore, the problems \textsc{$[\mathrm{IV} \ell \, | \, k \mathrm{A}]$-} and \textsc{$[\mathrm{IV} \ell \, | \, k \mathrm{V}]$-ConnectedComponents} are indeed NP-complete.
\end{proof}

Now we prove that given a vertex-colored digraph and a positive integer $m$, the problems of deciding if the digraph has a maximal $[\mathrm{IV} \ell \, | \, k \mathrm{A}]$- or $[\mathrm{IV} \ell \, | \, k \mathrm{V}]$-connected subgraph of size at least $m$ are both NP-complete -- even for undirected graphs in which every vertex has the same color.

\begin{proof}[Proof of Theorem~\ref{thm:complexity_of_IVlkA_and_IVlkV_2}] \label{proof:complexity_of_IVlkA_and_IVlkV_2}
 Let $k, \ell \in \mathbb{Z}_+$ be fixed. Obviously, these problems are in NP.
 
 Now we show that these problems are NP-hard by reducing the NP-complete problem \textsc{Clique}~\cite{karp1972reducibility} to them. Let $G$ be an arbitrary undirected graph. Let $G'$ be the undirected graph constructed in the proof of Theorem~\ref{thm:complexity_of_IVlkA_and_IVlkV}, i.e., $G'$ is obtained from $G$ by replacing each of its edges with $k$ parallel edges.
 Let $C = \{ c \}$ and let $f$ be the function that assigns the color $c$ to each vertex of $G'$. It is not difficult to see that the following are equivalent for any $W \subseteq V(G) = V(G')$ with $|W| \ge 2$:
 \begin{itemize}[topsep=3pt, itemsep=-1.5pt]
  \item[--] $G[W]$ is a clique,
  \item[--] the subgraph $G'[W]$ is $[\mathrm{IV} \ell \, | \, k \mathrm{A}]$-connected,
  \item[--] the subgraph $G'[W]$ is $[\mathrm{IV} \ell \, | \, k \mathrm{V}]$-connected.
 \end{itemize}
  
 Therefore, the problems \textsc{Maximal-$[\mathrm{IV} \ell \, | \, k \mathrm{A}]$-} and \textsc{Maximal-$[\mathrm{IV} \ell \, | \, k \mathrm{V}]$-ConnectedSubgraphs} are indeed NP-complete.
\end{proof}
\end{document}